\documentclass[leqno]{article}
\pdfoutput=1
\usepackage[cp1250]{inputenc}
\usepackage{authblk}
\usepackage[english]{babel}

\usepackage{a4wide}
 
\usepackage{amsfonts}
\usepackage{graphicx}
\usepackage{epstopdf}
\usepackage{bm}
\usepackage{subfig}
\usepackage{amsthm} 
\usepackage{amsmath}
\usepackage{adjustbox}

\newcommand{\grad}{\nabla}

\newcommand{\dx}{\ d \boldsymbol{x}}

\newcommand{\ds}{\ d s}

\newtheorem{lemma}{Lemma}
 \newtheorem{remark}{Remark}

\title{\textbf{Hybrid discontinuous Galerkin discretisation and domain decomposition preconditioners for the Stokes problem
}} \author[1]{Gabriel R. Barrenechea}
\author[1]{Micha\l{} Bosy \thanks{Corresponding author: Phone: +39 0382 985690, E-mail: michal.bosy@unipv.it, Fax: +39 0382 985602}}
\author[1]{Victorita Dolean}
\author[2]{Fr\'{e}d\'{e}ric Nataf}
\author[2]{Pierre-Henri Tournier}
\affil[1]{\small \textit{Department of Mathematics and Statistics, University of Strathclyde, 26 Richmond Street, G1 1XH Glasgow, United Kingdom}}
\affil[2]{\small \textit{Laboratory J.L. Lions, UPMC, CNRS UMR7598, 4 place Jussieu 75005 Paris, France}}
\affil[2]{\small \textit{INRIA Paris, EPC Alpines, 2 Rue Simone IFF, 75012 Paris, France}}

\date{\today}

\providecommand{\keywords}[1]{\textbf{Key words.} #1}

\begin{document}
\maketitle

\ 

\begin{abstract}
Solving the Stokes equation by an optimal domain decomposition method derived algebraically involves the use of non standard interface conditions whose discretisation is not trivial. For this reason the use of approximation methods such as hybrid discontinuous Galerkin appears as an appropriate strategy: on the one hand they provide the best compromise in terms of the number of degrees of freedom in between standard continuous and discontinuous Galerkin methods, and on the other hand the degrees of freedom used in the non standard interface conditions are naturally defined at the boundary between elements.In this paper we introduce the coupling between a well chosen discretisation method (hybrid discontinuous Galerkin) and a novel and efficient domain decomposition method to solve the Stokes system. We present the detailed analysis of the hybrid discontinuous Galerkin method for the Stokes problem with non standard boundary conditions. This analysis is supported by numerical evidence. In addition, the advantage of the new preconditioners over more classical choices is also supported by numerical experiments.
\end{abstract}


%
%
%

\section{Introduction}
\label{sec:introduction}

Discontinuous Galerkin (dG) methods have been first introduced in the early 1970s~\cite{reed1973triangular} and they have benefited of a
wide interest from the scientific community. The main advantages of these methods are their generality and flexibility as they can be used for a big variety of partial differential equations on unstructured meshes. Moreover, they can preserve local properties such as mass and momentum conservation while ensuring a high order accuracy. However, the cost of these advantages is a larger amount of
degrees of freedom in comparison to the continuous Galerkin methods~\cite{MR2050138} for the same approximation order.

A good compromise between the previous methods, while preserving the high order, are the hybridised versions of dG using divergence conforming spaces such as Raviart-Thomas (RT) and Brezzi-Douglas-Marini (BDM)~\cite{MR3097958}. These methods are a subset of the hybrid discontinuous Galerkin (hdG) methods introduced in~\cite{MR2485455} for second order elliptic problems. 

The hdG methods for the three-dimensional Stokes equation have been first introduced in \cite{MR2485446}. The authors present there the mixed formulation of hdG methods defined locally on each element. They consider many types of boundary conditions that involve normal and tangential
velocity, pressure, and tangential stress. The formulations of the methods are similar, the only difference is in the choice of the
numerical traces. The hdG methods for the Stokes equation with Dirichlet boundary conditions have been analysed in~\cite{MR2772094}
where the authors show the optimal convergence of the error for hdG methods and present different possibilities to obtain
superconvergence. On the other hand, in \cite{MR3047948} a hdG method for two or three dimensional Stokes equation with Dirichlet boundary conditions which is hybridisation of a symmetric interior penalty Galerkin method~\cite{MR1974180} is presented and analysed. In a further
paper~\cite{EggerW13} this approach is extended to Darcy, and coupled Darcy-Stokes flows. The new formulation includes different degrees of polynomials for finite element spaces associated with different variables. 

In~\cite{lehrenfeld2016high} the authors consider the Navier-Stokes problem, which can be seen as an extension of the Stokes and Stokes-Brinkman problems. To obtain the global mixed formulation, the authors choose $H(div)$-conforming finite elements.  Moreover, they introduce the formulation that includes a projection onto a space of lower polynomial degree. Such a modification allowed them to use fewer degrees of
freedom. In addition, it helped also to establish a connection between the hybrid high-order~\cite{MR3283758} and the hdG methods that presented authors of both methods in their joint paper~\cite{cockburn2015bridging}.

Despite that the hdG methods with projection allow us to decrease the number of degrees of freedom, nowadays most of the problems arise with
linear systems that are too big for direct solvers. Thus, parallel solvers are becoming increasingly important in scientific computing. A natural paradigm to take advantage of modern parallel architectures is domain decomposition method, see e.g. \cite{Smith:1996:DPM,Quarteroni:1999:DDM,Toselli:2005:DDM,MR3450068}. Domain decomposition methods are iterative solvers based on a decomposition of a global domain into subdomains. At each iteration, one (or two) boundary value problem(s) are solved in each subdomain and the continuity of the solution at the interfaces between subdomains is only satisfied at convergence of the iterative procedure. The partial differential equation is the one of the global problem.

For Additive Schwarz methods and Schur complement methods, the boundary conditions on the interfaces between subdomains, a.k.a. interface conditions (IC), are Dirichlet or Neumann boundary conditions. For Poisson problems, there is a consensus on these IC. But for systems of partial differential equations such as elasticity or Stokes problems, it has been envisioned that normal velocity-tangential flux (NVTF) or tangential velocity-normal flux (TVNF) IC should be superior to the pure velocity (Dirichlet like) or pure stress (Neumann like) IC, see \cite[Section~6.6]{MR3450068} and references therein. In \cite{Gosselet:2006:NDD}, it was motivated by symmetry considerations. In \cite{Dolean:NDD:09,Cluzeau:2012:STD,Cluzeau:2012:PTS}, they were obtained by an analysis of the systems of partial differential
equations by symbolic techniques mainly the Smith factorization~\cite{smith1861systems}. Similar attempts to derive more
intrinsic IC to the nature of the equation to solve were derived \cite{Dolean:NDD:06} for the Euler system.

Due to the difficulty of implementing these IC previous numerical tests were restricted to decompositions where boundaries of
subdomains are rectilinear so that the normal to the interface is easy
to define. The underlying domain decomposition method was a Schur
complement method. That is mainly the reason we have considered and
analysed a specific hdG method where this kind of degrees of freedom
are naturally present.

In this paper we want to combine appropriate hdG discretisation and
the associated domain decomposition methods mentioned above using
non standard IC. The combination of the two is
meant to provide the competitive solving strategy for this kind of partial differential equation
system. A different, but somewhat related, approach can be found in~\cite{de2014simple} where a dG type discretisation is coupled to a discrete Helmholtz decomposition to propose some preconditioners. 

An approach similar to the one presented in this work, but using completely discontinuous spaces, is given in~\cite{MR3486522}. Our analysis is related to the one in that paper, but the method presented herein uses $H(div)$-conforming spaces, which implies in turn that the Lagrange multipliers are scalar valued. The combination of these two facts reduces the number of degrees of freedom significantly. In addition, the use of non-standard boundary conditions (motivated by the newly defined domain decomposition preconditioners) makes the analysis somehow more involved.

The rest of the paper is organised as follows. To start with, we introduce the problem and notation in
Section~\ref{sec:notation}. In Section~\ref{sec:hdg} we present the
hybridisation of a symmetric interior penalty Galerkin method that
allows us to impose the TVNF and NVTF boundary
conditions in quite a natural way. The formulation is similar to the one
from~\cite{lehrenfeld2016high} with Dirichlet boundary conditions. In
addition to different kinds of boundary conditions, we included the projection to reduce the number of degrees of freedom. Our analysis follows the one from~\cite{lehrenfeld2016high} (see also~\cite{Lehrenfeldhesis} for a more detailed version). Thanks to the hdG discretisation, we can consider domain decomposition methods with arbitrary shape of the interfaces and Schwarz type methods. In Section~\ref{sec:dd_one_level}, the Additive Schwarz methods are defined at the algebraic level. Section~\ref{sec:numerics} contains the numerical results, including the convergence validation of the hdG method and a comparison of the domain decomposition preconditioners. Finally, some conclusions are drawn.


\section{Notation and preliminary results}
\label{sec:notation}

Let $\Omega$ be an open polygonal domain in $\mathbb{R}^2$ with Lipschitz boundary $\Gamma := \partial \Omega$. We use boldface font for tensor or vector variables e.g. $\boldsymbol{u}$ is a velocity vector field. The scalar variables will be italic e.g. $p$ denotes pressure scalar value. We define the stress tensor $\boldsymbol{\sigma} := \nu \grad \boldsymbol{u} - p \boldsymbol{I}$ and the flux as $\boldsymbol{\sigma_n} := \boldsymbol{\sigma} \ \boldsymbol{n}$. 
In addition we denote normal and tangential components as follows ${u_n} := \boldsymbol{u} \cdot \boldsymbol{n}$, $u_t := \boldsymbol{u}  \cdot \boldsymbol{t}$, $\sigma_{nn} := \boldsymbol{\sigma_n} \cdot \boldsymbol{n}$, $\sigma_{nt} := \boldsymbol{\sigma_n} \cdot \boldsymbol{t}$,
where $\boldsymbol{n}$ is the outward unit normal vector to the
boundary $\Gamma$ and $\boldsymbol{t}$ is a vector tangential to
$\Gamma$ such that $\boldsymbol{n} \cdot \boldsymbol{t} =
0$.

 For $D \subset \Omega$, we use the standard $L^2(D)$ space with following norm
\begin{eqnarray*}
\|f\|_D^2 := \int_D f^2 \dx & \mbox{for all } f \in L^2(D).
\end{eqnarray*}
Let us define following Sobolev spaces
\begin{align*}
H^m(D) & :=  \left\{v \in L^2(D): \ \forall \ |\boldsymbol{\alpha}| \leq m \ \partial^{\boldsymbol{\alpha}} {v} \in L^2(D)\right\} \mbox{ for } m \in \mathbb{N}, \\
H\left(div, D\right) & := \left\{\boldsymbol{v} \in [L^2(D)]^2: \ \nabla \cdot \boldsymbol{v} \in L^2(D)\right\},
\end{align*} 
where, for $\boldsymbol{\alpha} = (\alpha_1, \alpha_2) \in \mathbb{N}^2$ and
$|\boldsymbol{\alpha}| = \alpha_1+\alpha_2$ we denote $\partial^{\boldsymbol{\alpha}} = \frac{\partial^{|\boldsymbol{\alpha}|}}{\partial x_1^{\alpha_1} \partial x_2^{\alpha_2}}$.
In addition, we will use following standard semi-norm and norm for the Sobolev space $H^m(D)$ for $m \in \mathbb{N}$ 
\begin{align*}
|f|_{H^m(D)}^2 := \sum_{|\boldsymbol{\alpha}| = m} \|\partial^{\boldsymbol{\alpha}} f\|_D^2 && \|f\|_{H^m(D)}^2 := \sum_{k = 0}^m |f|_{H^k(D)}^2 && \mbox{for all } f \in H^m(D).
\end{align*}
In this work we consider the two dimensional Stokes problem:
\begin{equation}
\label{eq:stokes}
\left\{
\begin{array}{rclclr}
 -\nu \Delta \boldsymbol{u} & \ + & \grad p & = & \boldsymbol{f} & \mbox{in } \Omega, \\
 & & \nabla \cdot \boldsymbol{u} & = & 0 & \mbox{in } \Omega,
 \end{array}
\right.
\end{equation} 
where $\boldsymbol{u}: \bar{\Omega} \rightarrow \mathbb{R}^2$ is the unknown
velocity field, $p:\bar{\Omega} \rightarrow \mathbb{R}$ the pressure, $\nu >
0$ the viscosity which is considered to be constant and $\boldsymbol{f} \in [L^2(\Omega)]^2$ is a given function. 
For ${g} \in L^2(\Gamma)$ we consider two types of boundary conditions
\begin{itemize}
  \item tangential-velocity and normal-flux (TVNF) 
  \begin{equation}
  \left\{
  \begin{array}{rcll}
  \label{eq:TVNF}
   {\sigma_{nn}} & = & {g} & \mbox{ on } \Gamma, \\
   u_t & = & 0 & \mbox{ on } \Gamma,
  \end{array}
  \right. 
  \end{equation}
  \item normal-velocity and tangential-flux (NVTF) 
\begin{equation}
 \label{eq:NVTF}
 \left\{
 \begin{array}{rcll}
 {\sigma_{nt}} & = & {g} & \mbox{ on } \Gamma, \\
 {u_n} & = & 0 & \mbox{ on } \Gamma,
 \end{array}
 \right. 
\end{equation}
\end{itemize}
which together with \eqref{eq:stokes} define two boundary value
problems. We will detail the analysis for the TVNF boundary value problem
\begin{equation}
\label{eq:stokes_TVNF}
\left\{
\begin{array}{rclclr}
 -\nu \Delta \boldsymbol{u} & \ + & \grad p & = & \boldsymbol{f} & \mbox{in } \Omega, \\
 & & \nabla \cdot \boldsymbol{u} & = & 0 & \mbox{in } \Omega, \\
 & & {\sigma_{nn}} & = & {g} & \mbox{ on } \Gamma, \\
 & & u_t & = & 0 & \mbox{ on } \Gamma,
 \end{array}
\right.
\end{equation}
since considering the NVTF boundary conditions~\eqref{eq:NVTF} instead
is very similar. We will just add a remark when necessary to stress
the differences between them. The restriction to homogeneous Dirichlet conditions on $u_t$ is made only to simplify the presentation.

Let $\left\{\mathcal{T}_h\right\}_{h > 0}$ be a regular family of triangulations of $\bar{\Omega}$ made of triangles. For each triangulation $\mathcal{T}_h$, $\mathcal{E}_h$ denotes the set of its edges. In addition, for each of element $K \in \mathcal{T}_h$, $h_K := \mbox{diam}(K)$, and we denote $h := \max_{K \in \mathcal{T}_h} h_K$. 
We define following Sobolev spaces on the triangulation $\mathcal{T}_h$ and the set of all edges in $\mathcal{E}_h$
\begin{align*}
L^2(\mathcal{E}_h) & :=  \left\{v: \ v|_E \in L^2(E) \ \forall \ E \in \mathcal{E}_h \right\}, \\
H^m(\mathcal{T}_h) & :=  \left\{v \in L^2(\Omega): \ {v}|_K \in H^m(K) \ \forall \ K \in \mathcal{T}_h \right\} \mbox{ for } m \in \mathbb{N},
\end{align*}
with the corresponding broken norms.

 The following results will be very useful in what follows.

\begin{lemma}[Inverse and trace inequalities]
There exist $C, C_{max} > 0$, independent of $h_K$, such that for all $K \in \mathcal{T}_h$ and polynomial function $v$ in $K$ the following inequalities hold
\begin{align}
\label{l:local_inverse}
|v|_{H^s(K)} &\leq C h_K^{m-s} |v|_{H^m(K)}, \ 0 \leq m \leq s, \\
\label{l:discrete_trace_inverse}
h_K^{\frac{1}{2}} \|v\|_{\partial K} &\leq C_{max} \|{v}\|_{K}.
\end{align} 
Moreover, there exists $C > 0$, independent of $h_K$, such that for any $v \in H^1(K)$, the following local trace inequality holds
\begin{equation}
\label{l:local_trace}
\|v\|_{\partial K} \leq C \left(h_K^{-\frac{1}{2}} \|v\|_{K} + h_K^{\frac{1}{2}} |v|_{H^1(K)}\right).
\end{equation}
\end{lemma}
\begin{proof}
For~\eqref{l:local_inverse} see~\cite[Lemma 1.138]{MR2050138} and for~\eqref{l:discrete_trace_inverse} see~\cite[Lemma 1.46]{MR2882148}. The discrete trace inequality~\eqref{l:local_trace} follows by standard scaling arguments. 
\end{proof}
 Now we will introduce the finite element spaces that discretise the above spaces. Let us consider the TVNF boundary value problem~\eqref{eq:stokes_TVNF}. Let $k \geq 1$.
To discretise the velocity $\boldsymbol{u}$ we use the Brezzi-Douglas-Marini space (see~\cite[Section~2.3.1]{MR3097958}) 
\begin{align*}
 \boldsymbol{BDM_h^k} & := \left\{\boldsymbol{v_h} \in H\left(div, \Omega\right): \ \boldsymbol{v_h}|_K \in \left[\mathbb{P}_k\left(K\right)\right]^2 \ \forall \ {K \in \mathcal{T}_h}\right\}. 
\end{align*}
In addition, for $1 \leq m \leq k+1$ we denote $\Pi^k: [H^m(\Omega)]^2 \rightarrow \boldsymbol{BDM_h^k}$  the BDM projection defined in~\cite[Section~2.5]{MR3097958}.
 The hdG formulation includes a Lagrange multiplier over the internal edges. In order to propose a discretisation with fewer degrees of freedom, we discretise the Lagrange multiplier $\tilde{u}$ using the spaces
\begin{align*}
 M_{h}^{k-1} &:= \left\{\tilde{v}_h \in L^2\left(\mathcal{E}_h\right): \ \tilde{v}_h|_E \in \mathbb{P}_{k-1}\left(E \right) \ \forall \ {E \in \mathcal{E}_h} \right\}, \\ M_{h,0}^{k-1} &:= \left\{\tilde{v}_h \in M_{h}^{k-1}: \tilde{v}_h = 0 \mbox{ on } \Gamma\right\}. 
\end{align*} 
 Furthermore, we introduce for all $E \in \mathcal{E}_h$ the $L^2(E)$-projection $\Phi^{k-1}_E: L^2\left(E\right) \rightarrow \mathbb{P}_{k-1}\left(E\right)$ defined as follows. For every $\tilde{w} \in L^2\left(E\right)$, $\Phi^{k-1}_E(\tilde{w})$ is the unique element of $\mathbb{P}_{k-1}\left(E\right)$ satisfying 
\begin{equation}
\label{eq:L2edge_projection}
\int_{E} \Phi^{k-1}_E(\tilde{w}) {\tilde{v}_h} \ds = \int_{E} \tilde{w} {\tilde{v}_h} \ds \quad \forall \ {\tilde{v}_h} \in \mathbb{P}_{k-1}\left(E\right),
\end{equation}
and we denote $\Phi^{k-1}: L^2\left(\mathcal{E}_h\right) \rightarrow M_{h}^{k-1}$ defined as $\Phi^{k-1}|_E := \Phi^{k-1}_E$ for all $E \in \mathcal{E}_h$. 

 Let us denote $\boldsymbol{V_h} := \boldsymbol{BDM_h^k} \times M_{h,0}^{k-1}$. The pressure is discretised using the following space 
\begin{align*}
 Q_h^{k-1} & := \left\{q_h \in L^2\left(\Omega\right): \ q_h|_K \in \mathbb{P}_{k-1}\left(K\right) \ \forall \ {K \in \mathcal{T}_h}\right\} . 
\end{align*} 
In addition, we define the local $L^2(K)$-projection $\Psi_K^{k-1}: L^2(K) \rightarrow \mathbb{P}_{k-1}\left(K\right)$ for each $K \in \mathcal{T}_h$ defined as follows. For every $w \in L^2\left(K\right)$, $\Psi^{k}_K(w)$ is the unique element of $\mathbb{P}_{k-1}\left(K\right)$ satisfying
\begin{equation}
\label{eq:L2element_projection}
\int_{K} \Psi^{k-1}_K(w) v_h dx = \int_{K} {w v_h} dx \quad \forall \ v_h \in \mathbb{P}_{k-1}\left(K\right).
\end{equation}
 We will also use the following results.
\begin{lemma}[Approximation results]
There exists $C > 0$, independent of $h_K$, such that for all $\boldsymbol{v} \in [H^m(K)]^2$ and $v \in H^m(K)$, $1 \leq m \leq k+1$, the following interpolation estimates hold
\begin{itemize}
	\item local Brezzi-Douglas-Marini approximation
	\begin{align}
	\label{l:BDM_approximation}
	      \left\|\boldsymbol{v} - \Pi^k \left(\boldsymbol{v}\right)\right\|_{K} &\leq C h_K^m \left| \boldsymbol{v}\right|_{H^m\left(K\right)}, \\ \nonumber
	      \left\|\boldsymbol{v} - \Pi^k \left(\boldsymbol{v}\right)\right\|_{H^1\left(K\right)} &\leq C h_K^{m-1} \left| \boldsymbol{v}\right|_{H^m\left(K\right)},
	\end{align}
	\item trace $L^2$-projection approximation
	\begin{align}
	\label{l:trace_approximation}
	      \left\|v - \Phi^{k} \left(v\right)\right\|_{\partial K} &\leq C h_K^{m-\frac{1}{2}} \left|v\right|_{H^m\left(K\right)},
	\end{align}
	\item local $L^2$-projection approximation
	\begin{align}
	\label{l:L2_approximation}
	      \left\|v - \Psi^{k}_K \left(v\right)\right\|_{K} &\leq C h_K^m \left| v\right|_{H^m\left(K\right)}, \\ \nonumber
	      \left|v - \Psi^{k}_K \left(v\right)\right|_{H^1\left(K\right)} &\leq C h_K^{m-1} \left|v\right|_{H^m\left(K\right)}.
	\end{align}
\end{itemize}
\end{lemma}
\begin{proof}
For \eqref{l:BDM_approximation} see \cite[Preposition 2.5.1]{MR3097958}, for \eqref{l:trace_approximation} see \cite[Lemma III.2.10]{MR851383}, and for \eqref{l:L2_approximation} see the proof of \cite[Theorem 1.103]{MR2050138}.
\end{proof}


\section{Hybrid discontinuous Galerkin method}
\label{sec:hdg}

In this section we introduce the hdG method proposed in this work, study its well-posedness, and analyse its error.


\subsection{The discrete problem}
\label{sec:TVNF_formulation}

From now on we will use $\grad$ to denote the element-wise gradient. First, we multiply the first equation from~\eqref{eq:stokes} by a test function $\boldsymbol{v_h} \in \boldsymbol{BDM_h^k}$  and integrate by parts. This gives
\begin{align}
\label{eq:partial_integration}
 -\int_{\Omega} \nabla \cdot \left(\nu \grad \boldsymbol{u} \right) \boldsymbol{v_h} \dx + \int_{\Omega} \grad p \cdot \boldsymbol{v_h} \dx &= \sum_{K \in \mathcal{T}_h} \left(\int_K \nu \grad \boldsymbol{u} : \grad \boldsymbol{v_h} \dx - \int_K p \nabla \cdot \boldsymbol{v_h} \dx \right. \\ \nonumber
 & \left. - \int_{\partial K} \nu \boldsymbol{\partial_n u} \ \boldsymbol{v_h} \ds + \int_{\partial K} p \left(\boldsymbol{v_h}\right)_n \ds\right). 
\end{align}
Since the normal and tangential vectors are perpendicular ($\boldsymbol{n} \cdot \boldsymbol{t} = 0$) we can split~\eqref{eq:partial_integration} as
\begin{align}
\label{eq:orthogonality}
 -\int_{\Omega} \nabla \cdot \left(\nu \grad \boldsymbol{u} \right) \boldsymbol{v_h} \dx + \int_{\Omega} \grad p \cdot \boldsymbol{v_h} \dx &= \sum_{K \in \mathcal{T}_h} \left(\int_K \nu \grad \boldsymbol{u} : \grad \boldsymbol{v_h} \dx - \int_K p \nabla \cdot \boldsymbol{v_h} \dx\right. \\ \nonumber
& \left.  - \int_{\partial K} \sigma_{nt}  \left(\boldsymbol{v_h}\right)_t \ds - \int_{\partial K} \sigma_{nn}  \left(\boldsymbol{v_h}\right)_n \ds \right).
\end{align}
For the solution of the Stokes problem~\eqref{eq:stokes}, $\boldsymbol{\sigma_n}$ is continuous across all interior edges. Moreover, since $\boldsymbol{v_h} \in \boldsymbol{BDM_h^k}$, then $\left(\boldsymbol{v_h}\right)_n$ is continuous across all interior edges. Then we can rewrite~\eqref{eq:orthogonality} as follows
\begin{align}
\label{eq:sigma_continuity}
 -\int_{\Omega} \nabla \cdot \left(\nu \grad \boldsymbol{u} \right) \boldsymbol{v_h} \dx + \int_{\Omega} \grad p \cdot \boldsymbol{v_h} \dx &= \sum_{K \in \mathcal{T}_h} \left(\int_K \nu \grad \boldsymbol{u} : \grad \boldsymbol{v_h} \dx - \int_K p \nabla \cdot \boldsymbol{v_h} \dx  \right. \\ \nonumber
& \left.  - \int_{\partial K} \sigma_{nt}  \left(\boldsymbol{v_h}\right)_t \ds\right) - \int_{\Gamma} \sigma_{nn}  \left(\boldsymbol{v_h}\right)_n \ds.
\end{align}
Moreover, since $\boldsymbol{\sigma_n}$ is continuous across all interior edges, then $\sum_{K \in \mathcal{T}_h} \int_{\partial K} \sigma_{nt} \tilde{v}_h \ds = 0$, for all $\tilde{v}_h \in  M_{h,0}^{k-1}$, and we can add this to~\eqref{eq:sigma_continuity} to get
\begin{align}
\label{eq:lagrange_multiplier}
 -\int_{\Omega} \nabla \cdot \left(\nu \grad \boldsymbol{u} \right) \boldsymbol{v_h} \dx + \int_{\Omega} \grad p \cdot \boldsymbol{v_h} \dx &= \sum_{K \in \mathcal{T}_h} \left(\int_K \nu \grad \boldsymbol{u} : \grad \boldsymbol{v_h} \dx - \int_K p \nabla \cdot \boldsymbol{v_h} \dx  \right. \\ \nonumber
& \left.  - \int_{\partial K} \sigma_{nt}  \left(\left(\boldsymbol{v_h}\right)_t - \tilde{v}_h\right) \ds\right) - \int_{\Gamma} \sigma_{nn}  \left(\boldsymbol{v_h}\right)_n \ds.
\end{align}
Denoting $\tilde{u} = u_t$ on $\mathcal{E}_h$, then $\big(u_t - \tilde{u} \big) = \Phi^{k-1}\big(u_t - \tilde{u} \big) = 0$ on $\mathcal{E}_h$ and applying the boundary conditions~\eqref{eq:TVNF} we can rewrite~\eqref{eq:lagrange_multiplier} as
\begin{align}
\label{eq:viscous_part}
  -\int_{\Omega} \nabla \cdot \left(\nu \grad \boldsymbol{u} \right) \boldsymbol{v_h} \dx + \int_{\Omega} \grad p \cdot \boldsymbol{v_h} \dx &= \sum_{K \in \mathcal{T}_h} \left(\int_K \nu \grad \boldsymbol{u} : \grad \boldsymbol{v_h} \dx - \int_K p \nabla \cdot \boldsymbol{v_h} \dx  \right. \\ \nonumber
& \ \left.  - \int_{\partial K} \nu \left(\boldsymbol{\partial_n {u}}\right)_t  \left(\left(\boldsymbol{v_h}\right)_t - \tilde{v}_h\right) \ds \right. \\ \nonumber
& \ \pm \int_{\partial K} \nu \big(u_t - \tilde{u} \big) \left(\boldsymbol{\partial_n {v_h}}\right)_t \ds \\ \nonumber
 & \  \left. + \nu \frac{\tau}{h_K} \int_{\partial K} \Phi^{k-1}\big(u_t - \tilde{u} \big) \Phi^{k-1}\big(\left(\boldsymbol{v_h}\right)_t - \tilde{v}_h \big) \ds \right) \\ \nonumber
& \ - \int_{\Gamma} g  \left(\boldsymbol{v_h}\right)_n \ds,
\end{align}
where $\tau > 0$ is a stabilisation parameter.
 Hence, we define the velocity bilinear form $a: \boldsymbol{V_h} \times \boldsymbol{V_h} \rightarrow \mathbb{R}$ as
\begin{align}
\label{eq:TVNF_a_form}
  \nonumber a \left(\left(\boldsymbol{w_h}, {\tilde{w}_h}\right), \left(\boldsymbol{v_h}, {\tilde{v}_h}\right)\right) &:= \sum_{K \in \mathcal{T}_h} \left(\int_K \nu \grad \boldsymbol{w_h} : \grad \boldsymbol{v_h} \dx -
\int_{\partial K} \nu \left(\boldsymbol{\partial_n {w_h}}\right)_t \big(\left(\boldsymbol{v_h}\right)_t - \tilde{v}_h \big) \ds \right. \\
& \ + \varepsilon \int_{\partial K} \nu \big(\left(\boldsymbol{w_h}\right)_t - \tilde{w}_h \big) \left(\boldsymbol{\partial_n {v_h}}\right)_t \ds \\ \nonumber
 & \  \left. + \nu \frac{\tau}{h_K} \int_{\partial K} \Phi^{k-1}\big(\left(\boldsymbol{w_h}\right)_t - \tilde{w}_h \big) \Phi^{k-1}\big(\left(\boldsymbol{v_h}\right)_t - \tilde{v}_h \big) \ds \right) ,
\end{align}
where $\varepsilon \in \{-1,1\}$ and $\tau > 0$ is a stabilisation parameter 
and $b: \boldsymbol{V_h} \times Q_h^{k-1} \rightarrow \mathbb{R}$ as
\begin{equation}
\label{eq:TVNF_b_form}
 b\left(\left(\boldsymbol{v_h}, {\tilde{v}_h}\right), q_h\right) := - \sum_{K \in \mathcal{T}_h} \int_K q_h \nabla \cdot \boldsymbol{v_h} \dx.
\end{equation}
With these definitions we propose the hdG method for  the TVNF boundary value problem~\eqref{eq:stokes_TVNF}:\\
\textit{Find $\left(\boldsymbol{u_h}, {\tilde{u}_h}, p_h\right) \in \boldsymbol{V_h} \times Q^{k-1}_h$ such that for all $\left(\boldsymbol{v_h}, {\tilde{v}_h}, q_h\right) \in \boldsymbol{V_h}\times Q^{k-1}_h$}
\begin{equation}
\label{eq:TVNF_variational_formulation}
 \left\{
		\begin{array}{rclcl}
   a \left(\left(\boldsymbol{u_h}, {\tilde{u}_h}\right),\left(\boldsymbol{v_h}, {\tilde{v}_h}\right)\right) & \ + & b\left(\left(\boldsymbol{v_h}, {\tilde{v}_h}\right), p_h\right) & = & \displaystyle\int_{\Omega}\boldsymbol{f}\boldsymbol{v_h} \dx + \int_{\Gamma} {g} {\left(\boldsymbol{v_h}\right)_n} \ds\\[2ex]
		 & & b\left(\left(\boldsymbol{u_h}, {\tilde{u}_h}\right), q_h\right) & = & 0 .
   \end{array}
 \right.
\end{equation}

\begin{remark}

The use of $H(div)$-conforming spaces not only decrease the number of degrees of freedom in comparison to~\cite{MR3486522}, but lead as well to a simpler bilinear form $b$. 
\end{remark}

\subsection{Well-posedness of the discrete problem}
\label{sec:TVNF_exis_uniq}

Let us consider following semi-norm
\begin{align}
\label{eq:TVNF_norm}
|||\left(\boldsymbol{w_h}, {\tilde{w}_h}\right)|||^2 :=  \nu \sum_{K \in \mathcal{T}_h} \left(\left|\boldsymbol{w_h}\right|_{H^1(K)}^2 + h_K \left\| \boldsymbol{\partial_n w_h}\right\|_{\partial K}^2 +  \frac{\tau}{h_K} \left\|\Phi^{k-1}\big(\left(\boldsymbol{w_h}\right)_t - {\tilde{w}_h}\big)\right\|_{\partial K}^2\right). 
\end{align}
\begin{lemma}
\label{l:hdG_norm}
The semi-norm $||| \cdot |||$ defined by~\eqref{eq:TVNF_norm} is a norm on $\boldsymbol{V_h}$.
\end{lemma}
\begin{proof}
Since $||| \cdot |||$ is a semi-norm, we only need to show that
\begin{equation*}
|||\left(\boldsymbol{w_h}, {\tilde{w}_h}\right)||| = 0 \Rightarrow \boldsymbol{w_h} = \boldsymbol{0} \mbox{ and } \tilde{w}_h = 0.
\end{equation*}
Let us suppose $\left(\boldsymbol{w_h}, \tilde{w}_h\right) \in \boldsymbol{V_h}$ and $|||\left(\boldsymbol{w_h}, {\tilde{w}_h}\right)||| = 0$. Then $\grad \boldsymbol{w_h} = 0$ in all $K \in \mathcal{T}_h$, and thus $\boldsymbol{w_h} |_K = \boldsymbol{C_K}$ for all $K \in \mathcal{T}_h$. Now, since $\boldsymbol{w_h} \in [\mathbb{P}_0(K)]^2$ in every $K$
\begin{equation*}
\left\|\Phi^{k-1}\big(\left(\boldsymbol{w_h}\right)_t - {\tilde{w}_h}\big)\right\|_{\partial K} = 0 
\Rightarrow \left(\boldsymbol{w_h}\right)_t = {\tilde{w}_h} \mbox{ in each } E \in \mathcal{E}_h.
\end{equation*}
Since $\tilde{w}_h$ is single valued on all the edges in $\mathcal{E}_h$, then $\left(\boldsymbol{w_h}\right)_t$ is continuous in $\Omega$. Moreover, since $\boldsymbol{w_h}$ belongs to $\boldsymbol{BDM_h^k}$, $\left(\boldsymbol{w_h}\right)_n$ is also continuous in $\Omega$. Then, $\boldsymbol{w_h}$ is continuous in $\Omega$, and thus $\boldsymbol{w_h} = \boldsymbol{C} \in \mathbb{R}^2$ in $\Omega$. Finally, since
\begin{equation*}
	\left(\boldsymbol{w_h}\right)_t = \left(\boldsymbol{C}\right)_t = 0 \mbox{ on } \Gamma \Rightarrow \boldsymbol{w_h} = \boldsymbol{0} \mbox{ in } \Omega,
\end{equation*}
which finishes the proof since ${\tilde{w}_h} = \left(\boldsymbol{w_h}\right)_t$ on every edge.
\end{proof}

\begin{lemma}
\label{l:TVNF_continuity}
There exists $C > 0$ such that, for all $\left(\boldsymbol{w}, \tilde{w}\right)$, $\left(\boldsymbol{v}, \tilde{v}\right) \in \left[H^1\left(\Omega\right) \cap H^2\left(\mathcal{T}_h\right)\right]^2 \times L^2\left(\mathcal{E}_h\right)$ and $q \in L^2\left(\Omega\right)$, we have
\begin{align}
\label{eq:continuity_a}
|a\left(\left(\boldsymbol{w}, \tilde{w}\right), \left(\boldsymbol{v}, \tilde{v}\right)\right)| & \leq C |||\left(\boldsymbol{w}, \tilde{w}\right)||| \ |||\left(\boldsymbol{v}, \tilde{v}\right)|||, \\
\label{eq:continuity_b}
|b\left(\left(\boldsymbol{w}, \tilde{w}\right),q\right)| & \leq \sqrt{\frac{2}{\nu}} |||\left(\boldsymbol{w}, \tilde{w}\right)||| \left\|q\right\|_{\Omega}.
\end{align}
\end{lemma}
\begin{proof}
Let us start with~\eqref{eq:continuity_a}. Using the Cauchy-Schwarz inequality we get
\begin{align*}
|a\left(\left(\boldsymbol{w}, \tilde{w}\right), \left(\boldsymbol{v}, \tilde{v}\right)\right)| & \leq 2 |||\left(\boldsymbol{w}, \tilde{w}\right)||| \ |||\left(\boldsymbol{v}, \tilde{v}\right)||| \\
& \ + \sum_{K \in \mathcal{T}_h} \left(\nu\left\| \boldsymbol{\partial_n w}\right\|_{\partial K} \left\|{v}_t - \tilde{v}\right\|_{\partial K} + \nu\left\| \boldsymbol{\partial_n v}\right\|_{\partial K} \left\|{w}_t - \tilde{w}\right\|_{\partial K}\right).
\end{align*}
Therefore, using the triangle inequality and the trace $L^2$-projection approximation~\eqref{l:trace_approximation} we get
\begin{align}
 \nonumber
\left\| \boldsymbol{\partial_n w}\right\|_{\partial K} \left\|{v}_t - \tilde{v}\right\|_{\partial K} & \leq \left\| \boldsymbol{\partial_n w}\right\|_{\partial K} \left\|{v}_t - \Phi^{k-1}\left({v}_t\right)\right\|_{\partial K}+ \left\| \boldsymbol{\partial_n w}\right\|_{\partial K} \left\|\Phi^{k-1}\left({v}_t - \tilde{v}\right)\right\|_{\partial K}\\
\label{eq:continuity_triangle_inequality}
& \leq \sqrt{h_K} \left\| \boldsymbol{\partial_n w}\right\|_{\partial K} \left(\tilde{c}_1 \left|\boldsymbol{v} \right|_{H^1(K)} + \frac{1}{\sqrt{h_K}}\left\|\Phi^{k-1}\left({v}_t - \tilde{v}\right)\right\|_{\partial K}\right).
\end{align}
Thus, using the Cauchy-Schwarz inequality
\begin{align*}
\nu \left\| \boldsymbol{\partial_n w}\right\|_{\partial K} \left\|v_t - \tilde{v}\right\|_{\partial K} & \leq c_1 |||\left(\boldsymbol{w}, \tilde{w}\right)||| \ |||\left(\boldsymbol{v}, \tilde{v}\right)|||, \\
\nu \left\|\boldsymbol{\partial_n {v}}\right\|_{\partial K} \left\|w_t - \tilde{w}\right\|_{\partial K} & \leq c_2 |||\left(\boldsymbol{v}, \tilde{v}\right)||| \ |||\left(\boldsymbol{w}, \tilde{w}\right)|||.
\end{align*}
Finally, we get~\eqref{eq:continuity_a} for $C = \left(2 + c_1 +
  c_2\right)$. The continuity~\eqref{eq:continuity_b} is analogous. 
\end{proof}
 To show the well-posedness of~\eqref{eq:TVNF_variational_formulation}
we need the ellipticity of the bilinear form $a$ and an inf-sup
condition for the bilinear form $b$. We start by showing that $a$ is elliptic with respect to $||| \cdot |||$.
\begin{lemma}
\label{l:TVNF_coercivity}
There exists $\alpha > 0$ such that for all $\left(\boldsymbol{v_h}, {\tilde{v}_h}\right) \in \boldsymbol{V_h}$
\begin{equation}
\label{eq:coercivity_a}
a\left(\left(\boldsymbol{v_h}, {\tilde{v}_h}\right), \left(\boldsymbol{v_h}, {\tilde{v}_h}\right)\right) \geq \alpha |||\left(\boldsymbol{v_h}, {\tilde{v}_h}\right)|||^2.
\end{equation}
If $\varepsilon = -1$ in the definition~\eqref{eq:TVNF_a_form}, then this only holds under the additional hypothesis of $\tau$ being large enough. If $\varepsilon = 1$ in~\eqref{eq:TVNF_a_form}, this inequality holds for arbitrary $\tau$.
\end{lemma}
\begin{proof}
First, since $\boldsymbol{\partial_n v_h}|_E \in [\mathbb{P}_{k-1}(E)]^2$ for all $E \in \mathcal{E}_h$, then
\begin{align}
\label{eq:coervity_projection}
a\left(\left(\boldsymbol{v_h}, {\tilde{v}_h}\right), \left(\boldsymbol{v_h}, {\tilde{v}_h}\right)\right) & = \sum_{K \in \mathcal{T}_h} \left( \nu \left|\boldsymbol{v_h} \right|_{H^1(K)}^2 -
 \nu \left(1- \varepsilon\right) \int_{\partial K} \left(\boldsymbol{\partial_n {v_h}}\right)_t \Phi^{k-1}\big(\left(\boldsymbol{v_h}\right)_t - \tilde{v}_h \big) \ds \right. \\ \nonumber
& \ \left. + \nu \frac{\tau}{h_K} \left\|\Phi^{k-1}\left(\left(\boldsymbol{v_h}\right)_t - {\tilde{v}_h}\right)\right\|_{\partial K}^2\right) .
\end{align}
To bound the middle term in terms of the other two, we consider two cases.

  $\bullet$ if $\varepsilon = 1$, then~\eqref{eq:coervity_projection} reduces to
\begin{equation}
\label{eq:coercivity_a_non_sym}
a\left(\left(\boldsymbol{v_h}, {\tilde{v}_h}\right), \left(\boldsymbol{v_h}, {\tilde{v}_h}\right)\right) = \sum_{K \in \mathcal{T}_h}\left( \nu \left|\boldsymbol{v_h} \right|_{H^1(K)}^2 + \nu \frac{\tau}{h_K} \left\|\Phi^{k-1}\left(\left(\boldsymbol{v_h}\right)_t - {\tilde{v}_h}\right)\right\|_{\partial K}^2\right).
\end{equation}
It only remains to show that the right hand side of~\eqref{eq:coercivity_a_non_sym} is an upper bound (up to a constant) for the norm $||| \cdot |||$ given by~\eqref{eq:TVNF_norm}. Using the discrete trace inequality~\eqref{l:discrete_trace_inverse} we get
\begin{align*}
\sum_{K \in \mathcal{T}_h} h_K \left\| \boldsymbol{\partial_n v_h}\right\|_{\partial K}^2 &\leq \sum_{K \in \mathcal{T}_h} C_{max}^2 \left|\boldsymbol{v_h}\right|_{H^1(K)}^2,
\end{align*}
and then
\begin{equation}
\label{eq:TVNF_norm_relation}
|||\left(\boldsymbol{v_h}, {\tilde{v}_h}\right)|||^2 \leq \left(1+C_{max}^2\right) \sum_{K \in \mathcal{T}_h} \nu \left(\left|\boldsymbol{v_h} \right|_{H^1(K)}^2 + \frac{\tau}{h_K} \left\|\Phi^{k-1}\left(\left(\boldsymbol{v_h}\right)_t - {\tilde{v}_h}\right)\right\|_{\partial K}^2\right),
\end{equation}
which proves~\eqref{eq:coercivity_a} with $\alpha = \frac{1}{1+C_{max}^2}$. \\
  $\bullet$ if $\varepsilon = -1$, then~\eqref{eq:coervity_projection} becomes
\begin{align*}
a\left(\left(\boldsymbol{v_h}, {\tilde{v}_h}\right), \left(\boldsymbol{v_h}, {\tilde{v}_h}\right)\right) & = \sum_{K \in \mathcal{T}_h} \left( \nu \left|\boldsymbol{v_h} \right|_{H^1(K)}^2 -
 2\nu \int_{\partial K} \left(\boldsymbol{\partial_n {v_h}}\right)_t \Phi^{k-1}\big(\left(\boldsymbol{v_h}\right)_t - \tilde{v}_h \big) \ds \right. \\ 
& \ \left. + \nu \frac{\tau}{h_K} \left\|\Phi^{k-1}\left(\left(\boldsymbol{v_h}\right)_t - {\tilde{v}_h}\right)\right\|_{\partial K}^2\right) .
\end{align*}
Using the Cauchy-Schwarz inequality
\begin{align*}
a\left(\left(\boldsymbol{v_h}, {\tilde{v}_h}\right), \left(\boldsymbol{v_h}, {\tilde{v}_h}\right)\right) & \geq \sum_{K \in \mathcal{T}_h} \left(\nu\left|\boldsymbol{v_h} \right|_{H^1(K)}^2 -
 2\nu \left\| \boldsymbol{\partial_n v_h}\right\|_{\partial K} \left\|\Phi^{k-1}\left(\left(\boldsymbol{v_h}\right)_t - {\tilde{v}_h}\right)\right\|_{\partial K} \right.\\
& \ \left. + \nu\frac{\tau}{h_K} \left\|\Phi^{k-1}\left(\left(\boldsymbol{v_h}\right)_t - {\tilde{v}_h}\right)\right\|_{\partial K}^2\right).
\end{align*}
Since $\boldsymbol{v_h} \in \boldsymbol{BDM_h^k}$ is a piecewise polynomial we can apply the discrete trace inequality~\eqref{l:discrete_trace_inverse} to the second term, followed by the Young's inequality to arrive at 
\begin{align*}
a\left(\left(\boldsymbol{v_h}, {\tilde{v}_h}\right), \left(\boldsymbol{v_h}, {\tilde{v}_h}\right)\right) &\geq \sum_{K \in \mathcal{T}_h} \left(\nu \left|\boldsymbol{v_h} \right|_{H^1(K)}^2 -
 2\nu \frac{C_{max}}{\sqrt{h_K}}\left|\boldsymbol{v_h}\right|_{H^1(K)} \left\|\Phi^{k-1}\left(\left(\boldsymbol{v_h}\right)_t - {\tilde{v}_h}\right)\right\|_{\partial K} \right. \\
& \ \left.+ \nu\frac{\tau}{h_K} \left\|\Phi^{k-1}\left(\left(\boldsymbol{v_h}\right)_t - {\tilde{v}_h}\right)\right\|_{\partial K}^2\right) \\
& \geq 
\sum_{K \in \mathcal{T}_h} \left(\frac{\nu}{2} \left|\boldsymbol{v_h}\right|_{H^1(K)}^2 + \nu\frac{\tau - 2 C_{max}^2}{h_K} \left\|\Phi^{k-1}\left(\left(\boldsymbol{v_h}\right)_t - {\tilde{v}_h}\right)\right\|_{\partial K}^2 \right) \\
& \geq \nu C \sum_{K \in \mathcal{T}_h}\left(  \left|\boldsymbol{v_h} \right|_{H^1(K)}^2 + \frac{\tau}{h_K} \left\|\Phi^{k-1}\left(\left(\boldsymbol{v_h}\right)_t - {\tilde{v}_h}\right)\right\|_{\partial K}^2\right).
\end{align*}
Finally, if we suppose $\tau > 2 C_{max}^2$ then $C := \min\left\{\frac{1}{2},\frac{\tau - 2 C_{max}^2}{\tau}\right\} > 0$,
using~\eqref{eq:TVNF_norm_relation} we get~\eqref{eq:coercivity_a} for $\alpha =  \frac{C}{1+C_{max}^2}$.
\end{proof}
 The next step towards stability is proving the inf-sup condition for $b$, which is done next.
\begin{lemma}
\label{l:TVNF_inf_sup}
There exists $\beta > 0$ independent of $h_K$ such that
\begin{equation*}
\sup_{\left(\boldsymbol{v_h}, {\tilde{v}_h}\right) \in \boldsymbol{V_h}} \frac{b\left(\left(\boldsymbol{v_h}, {\tilde{v}_h}\right), q_h\right)}{|||\left(\boldsymbol{v_h}, {\tilde{v}_h}\right)|||} \geq \frac{\beta}{\sqrt{\nu}} \left\|q_h\right\|_{\Omega} \ \forall q_h \in Q_h^{k-1}.
\end{equation*}
\end{lemma}
\begin{proof}
According to the Fortin criterion, see~\cite[Lemma~4.19]{MR2050138}, we need to prove that there exists a Fortin operator $\boldsymbol{\Pi}: \left[H^1\left(\Omega\right) \right]^2 \rightarrow \boldsymbol{V_h}$ such that for every $\boldsymbol{v} \in [H^1(\Omega)]^2$ the following conditions hold
\begin{align}
\label{eq:fortin_operator_stokes_first}
 b\left(\left(\boldsymbol{v}, \tilde{v}\right), q_h \right) & = b\left(\boldsymbol{\Pi}\left(\boldsymbol{v}\right), q_h \right) \quad \forall \ {q_h \in Q_h^{k-1}}, \\
\label{eq:fortin_operator_stokes_second}
 |||\boldsymbol{\Pi}\left(\boldsymbol{v}\right)||| & \leq C  \sqrt{\nu}\|\boldsymbol{v}\|_{H^1\left(\Omega\right)}.
\end{align}
Let $\boldsymbol{v} \in [H^1(\Omega)]^2$ and let us consider the operator $\boldsymbol{\Pi}\left(\boldsymbol{v}\right) := \left(\Pi^k \left(\boldsymbol{v}\right),\Phi^{k-1} \left({v}_t\right)\right)$.
It is well known, see~\cite[Section~2.5]{MR3097958}, that $\Pi^k$ satisfies
~\eqref{eq:fortin_operator_stokes_first}.
To prove \eqref{eq:fortin_operator_stokes_second} we denote $\left(\boldsymbol{w_h}, {\tilde{w}_h}\right) := \boldsymbol{\Pi} \left(\boldsymbol{v}\right)$.
Then using the discrete trace inequality~\eqref{l:discrete_trace_inverse} and the fact that the projection is a bounded operator, we get
\begin{align}
\label{eq:bound_operator}
\nonumber ||| \left(\boldsymbol{w_h}, {\tilde{w}_h}\right) |||^2 & = \sum_{K \in \mathcal{T}_h} \nu \left(\left|\boldsymbol{w_h}\right|_{H^1(K)}^2 + h_K \left\| \boldsymbol{\partial_n w_h}\right\|_{\partial K}^2 + \frac{\tau}{h_K} \left\|\Phi^{k-1} \left(\left(\boldsymbol{w_h}\right)_t - {\tilde{w}_h}\right)\right\|_{\partial K}^2 \right) \\
& \leq \sum_{K \in \mathcal{T}_h}\nu \left( \left(1+C_{max}^2\right)  \left|\boldsymbol{w_h}\right|_{H^1(K)}^2 + \frac{\tau}{h_K} \left\|\left(\boldsymbol{w_h}\right)_t - {\tilde{w}_h}\right\|_{\partial K}^2 \right).
\end{align}
Applying the triangle inequality for the last term of~\eqref{eq:bound_operator} we arrive at
\begin{align}
\label{eq:inf_sup_terms}
\nonumber |||\left(\boldsymbol{w_h}, {\tilde{w}_h}\right) |||^2 & \leq \sum_{K \in \mathcal{T}_h} \nu \left(\left(1+C_{max}^2\right) \left|\boldsymbol{w_h}\right|_{H^1(K)}^2 + \frac{2\tau}{h_K} \left( \left\|\left(\boldsymbol{w_h}\right)_t - {v}_t\right\|_{\partial K}^2 + \left\|{v}_t - {\tilde{w}_h}\right\|_{\partial K}^2 \right) \right) \\
& =: \sum_{K \in \mathcal{T}_h}  \nu \left(\left(1+C_{max}^2\right) \mathfrak{T}_1^K + \frac{2\tau}{h_K} \left(\mathfrak{T}_2^K + \mathfrak{T}_3^K\right) \right). 
\end{align}
Using the stability of $\Pi^k$ we get
\begin{equation}
\label{eq:first_term_inf_sup}
\mathfrak{T}_1^K = \left|\Pi^k \left(\boldsymbol{v}\right)\right|_{H^1\left(K\right)}^2 \leq {c_1} \left| \boldsymbol{v}\right|_{H^1\left(K\right)}^2.
\end{equation}
Using~\eqref{l:BDM_approximation} and the local trace inequality~\eqref{l:local_trace}, then 
\begin{align}
\label{eq:second_term_inf_sup}
  \nonumber
\mathfrak{T}_2^K \leq \left\|\boldsymbol{v} - \boldsymbol{w_h}\right\|_{\partial K}^2 
& \leq \tilde{c}_1 \left(\frac{1}{h_K} \left\|\boldsymbol{v} - \boldsymbol{w_h}\right\|_{K}^2 + h_K \left|\boldsymbol{v} - \boldsymbol{w_h}\right|_{H^1\left(K\right)}^2 \right) \\
& \leq \tilde{c}_1 \left(\tilde{c}_2 h_K \left|\boldsymbol{v}\right|_{H^1\left(K\right)}^2 + \tilde{c}_3 h_K \left|\boldsymbol{v}\right|_{H^1\left(K\right)}^2 \right) \leq \tilde{c}_1 \left(\tilde{c}_2 + \tilde{c}_3\right) h_K \left|\boldsymbol{v}\right|_{H^1\left(K\right)}^2.
\end{align}
Finally, using the trace $L^2$-projection approximation~\eqref{l:trace_approximation} for the third term we get
\begin{align}
\label{eq:third_term_inf_sup}
\mathfrak{T}_3^K & \leq \tilde{c}_4  h_K \left|\boldsymbol{v}\right|_{H^1\left(K\right)}^2.
\end{align}
Then collecting \eqref{eq:first_term_inf_sup}, \eqref{eq:second_term_inf_sup} and \eqref{eq:third_term_inf_sup}, we obtain~\eqref{eq:fortin_operator_stokes_second} with
\begin{equation*}
C :=  \sqrt{\Big(\left(1+C_{max}^2\right)c_1 + 2\tau \tilde{c}_1 \left(\tilde{c}_2 + \tilde{c}_3\right) + 2\tau \tilde{c}_4\Big)}, 
\end{equation*}
which finishes the proof.
\end{proof}
Using the last two results and the standard Babuska-Brezzi's results~\cite[Section~4.2]{MR3097958} we deduce there exists a unique solution of~\eqref{eq:TVNF_variational_formulation}. In addition, method~\eqref{eq:TVNF_variational_formulation} is consistent that the following result shows.
\begin{lemma}[Consistency]
\label{l:TVNF_consistency}
 Let $\left(\boldsymbol{u},p\right) \in \left[H^1\left(\Omega\right) \cap H^2\left(\mathcal{T}_h\right)\right]^2 \times L^2\left(\Omega\right)$ be the solution of the problem \eqref{eq:stokes_TVNF} and $\tilde{u} = u_t$ on all edges of $\mathcal{E}_h$. If $\left(\boldsymbol{u_h}, {\tilde{u}_h}, p_h \right) \in \boldsymbol{V_h} \times Q^{k-1}_h$ solves \eqref{eq:TVNF_variational_formulation}, then for all $\left(\boldsymbol{v_h}, {\tilde{v}_h}, q_h\right) \in \boldsymbol{V_h} \times Q^{k-1}_h$ the following holds
\begin{equation*}
 a\left(\left(\boldsymbol{u}-\boldsymbol{u_h},\tilde{u}- {\tilde{u}_h}\right),\left(\boldsymbol{v_h}, {\tilde{v}_h}\right)\right)+b\left(\left(\boldsymbol{u}-\boldsymbol{u_h},\tilde{u}- {\tilde{u}_h}\right), q_h\right)+b\left(\left(\boldsymbol{v_h}, {\tilde{v}_h}\right),p- p_h\right) = 0.
 \end{equation*}
\end{lemma}
\begin{proof}
As we have seen in Section~\ref{sec:TVNF_formulation}, all added terms are zero for $\left(\boldsymbol{u}, \tilde{u}\right)$. Thus
\begin{equation*}
\left\{
\begin{array}{rclcl}
a\left(\left(\boldsymbol{u},\tilde{u}\right),\left(\boldsymbol{v_h}, {\tilde{v}_h}\right)\right)&+&b\left(\left(\boldsymbol{v_h}, {\tilde{v}_h}\right),p\right) &=& \displaystyle\int_{\Omega}\boldsymbol{f}\boldsymbol{v_h} \dx + \int_{\Gamma}g\left(\boldsymbol{v_h}\right)_n \ds \\
&&b\left(\left(\boldsymbol{u},\tilde{u}\right), q_h\right) &=& 0
\end{array}
\right. ,
 \end{equation*}
which proves the result.
\end{proof}


\subsection{Error analysis}
\label{sec:TVNF_error}
In this section we present the error estimates for the method. These estimates are proved using the following norm
\begin{equation}
	|||(\boldsymbol{u}, \tilde{u}, p)|||_h := |||(\boldsymbol{u}, \tilde{u})||| + \frac{1}{\sqrt{\nu}}\|p\|_{\Omega}.
\end{equation}
 The first step is the following version of Cea's lemma.

\begin{lemma}
\label{l:TVNF_Cea}
Let $\left(\boldsymbol{u},p\right) \in \left[H^1\left(\Omega\right) \cap H^2\left(\mathcal{T}_h\right)\right]^2 \times L^2\left(\Omega\right)$ be the solution of~\eqref{eq:stokes_TVNF}, $\tilde{u} = u_t$ on all edges in $\mathcal{E}_h$, and $\left(\boldsymbol{u_h}, {\tilde{u}_h}, p_h \right) \in \boldsymbol{V_h} \times Q^{k-1}_h$ these of~\eqref{eq:TVNF_variational_formulation}. Then there exists $C > 0$, independent of $h$ and $\nu$, such that
\begin{equation}
\label{eq:cea}
|||\left(\boldsymbol{u}-\boldsymbol{u_h},\tilde{u}- {\tilde{u}_h}, p - p_h\right)|||_h \leq C \inf_{\left(\boldsymbol{v_h}, {\tilde{v}_h}, q_h\right) \in \boldsymbol{V_h} \times Q^{k-1}_h} |||\left(\boldsymbol{u}-\boldsymbol{v_h},\tilde{u}- {\tilde{v}_h},p - q_h\right)|||_h.
\end{equation}
\end{lemma}
\begin{proof}
Let us denote
\begin{equation*}
B \left(\left(\boldsymbol{w_h}, {\tilde{w}_h}, r_h \right),\left(\boldsymbol{v_h}, {\tilde{v}_h}, q_h\right)\right) := a \left(\left(\boldsymbol{w_h}, {\tilde{w}_h}\right),\left(\boldsymbol{v_h}, {\tilde{v}_h}\right)\right) + b\left(\left(\boldsymbol{v_h}, {\tilde{v}_h}\right), r_h \right) + b\left(\left(\boldsymbol{w_h}, {\tilde{w}_h}\right), q_h\right).
\end{equation*}
Using Lemmas~\ref{l:TVNF_coercivity} and~\ref{l:TVNF_inf_sup}, and~\cite[Preposition 2.36]{MR2050138}, we get the following stability for $B$. \\
There exists $\beta_B > 0$, independent of $h$ and $\nu$, such that for all $\left(\boldsymbol{v_h}, {\tilde{v}_h}, q_h \right) \in \boldsymbol{V_h} \times Q^{k-1}_h$ there exists $\left(\boldsymbol{w_h}, {\tilde{w}_h}, r_h\right) \in \boldsymbol{V_h} \times Q^{k-1}_h$ such that $|||\left(\boldsymbol{w_h}, {\tilde{w}_h},r_h\right)|||_h=1$, and 
\begin{equation}
\label{eq:inf_sup_l}
B \left(\left(\boldsymbol{v_h}, {\tilde{v}_h}, q_h\right),\left(\boldsymbol{w_h}, {\tilde{w}_h}, r_h \right)\right) \geq \beta_B |||\left(\boldsymbol{v_h}, {\tilde{v}_h}, q_h\right)|||_h .
\end{equation}
Now using Lemma~\ref{l:TVNF_continuity}, we get continuity of $B$, there exists $C_B > 0$
\begin{align}
\label{eq:continuity_l}
\left|B \left(\left(\boldsymbol{w_h}, {\tilde{w}_h}, r_h \right),\left(\boldsymbol{v_h}, {\tilde{v}_h}, q_h\right)\right)\right| & \leq 
C_B |||\left(\boldsymbol{w_h}, {\tilde{w}_h}, r_h\right)|||_h |||\left(\boldsymbol{v_h}, {\tilde{v}_h},q_h\right)|||_h.
\end{align}
Let $\left(\boldsymbol{v_h}, {\tilde{v}_h}, q_h\right) \in \boldsymbol{V_h}$. 
Then, using Lemma~\ref{l:TVNF_consistency}, the triangle inequality, \eqref{eq:inf_sup_l} and~\eqref{eq:continuity_l} we arrive at
\begin{align*}
|||\left(\boldsymbol{v_h} - \boldsymbol{u_h},\tilde{v}_h- \tilde{u}_h,q_h-p_h\right)|||_h & \leq \frac{1}{\beta_B} B \left(\left(\boldsymbol{v_h} - \boldsymbol{u}, {\tilde{v}_h}- \tilde{u}, q_h-p\right),\left(\boldsymbol{w_h}, {\tilde{w}_h}, r_h\right)\right) \\
& + \frac{1}{\beta_B} B \left(\left(\boldsymbol{u}-\boldsymbol{u_h},\tilde{u}- {\tilde{u}_h},p- p_h\right),\left(\boldsymbol{w_h}, {\tilde{w}_h}, r_h\right)\right) \\
& \leq \frac{C_B}{\beta_B} |||\left(\boldsymbol{v_h} - \boldsymbol{u},{\tilde{v}_h}- \tilde{u},q_h-p\right)|||_h.
\end{align*}
Thus, we get~\eqref{eq:cea} with $C := 1+\frac{C_B}{\beta_B}$.
\end{proof}
 Using standard interpolation estimates, the following error estimate is proved.
\begin{lemma}[hdG error]
\label{l:TVNF_hdG_error}
 Let us assume $\left(\boldsymbol{u},p\right) \in \left[H^1\left(\Omega\right) \cap H^{k+1}\left(\mathcal{T}_h\right)\right]^2 \times H^k\left(\mathcal{T}_h\right)$ is the solution of~\eqref{eq:stokes_TVNF}, and $\tilde{u} = u_t$ on all edges in $\mathcal{E}_h$. If $\left(\boldsymbol{u_h}, {\tilde{u}_h}, p_h \right) \in \boldsymbol{V_h} \times Q^{k-1}_h$ solves the discrete problem~\eqref{eq:TVNF_variational_formulation}, then there exists $C > 0$, independent of $h$, such that
\begin{equation}
\label{eq:hdG_error}
|||\left(\boldsymbol{u}-\boldsymbol{u_h},\tilde{u}- {\tilde{u}_h},p - p_h\right)|||_h  \leq C h^k \left(\sqrt{\nu} \|\boldsymbol{u}\|_{H^{k+1}\left(\mathcal{T}_h\right)} +\frac{1}{\sqrt{\nu}}\|p\|_{H^k\left(\mathcal{T}_h\right)} \right).
\end{equation}
\end{lemma}
\begin{proof}
Let us consider the Fortin operator $\boldsymbol{\Pi}$ defined in the proof of Lemma~\ref{l:TVNF_inf_sup}.  If $\boldsymbol{\Pi}\left(\boldsymbol{u}\right) = \left(\boldsymbol{w_h}, {\tilde{w}_h}\right)$, then by using the triangle inequality and boundedness of the projection~$\Phi^{k-1}$ we get
\begin{align}
\label{eq:hdG_error_triangle_inequality}
 |||\left(\boldsymbol{u}-\boldsymbol{w_h},\tilde{u}- {\tilde{w}_h}\right)|||^2 & = \sum_{K\in \mathcal{T}_h} \nu \left(\left|\boldsymbol{u} - \boldsymbol{w_h}\right|_{H^1(K)}^2 + h_K \left\|\partial_{\boldsymbol{n}} \left(\boldsymbol{u} - \boldsymbol{w_h}\right)\right\|_{\partial K}^2 \right. \\ \nonumber
& \ \left. + \frac{\tau}{h_K} \left\|\Phi^{k-1} \big(\left(\boldsymbol{u} - \boldsymbol{w_h}\right)_t - \left(\tilde{u} - {\tilde{w}_h}\right)\big)\right\|_{\partial K}^2 \right)\\ \nonumber
& \leq \sum_{K\in \mathcal{T}_h} \nu \left(\left|\boldsymbol{u} - \boldsymbol{w_h}\right|_{H^1(K)}^2 + h_K \left\|\partial_{\boldsymbol{n}} \left(\boldsymbol{u} - \boldsymbol{w_h}\right)\right\|_{\partial K}^2 \right. \\ \nonumber
& \ \left. + \frac{2c_1 \tau}{h_K} \left(\left\|\boldsymbol{u} - \boldsymbol{w_h} \right\|_{\partial K}^2 + \left\|\tilde{u} - {\tilde{w}_h}\right\|_{\partial K}^2\right) \right) \\ \nonumber
& =: \sum_{K\in \mathcal{T}_h} \nu \left( \mathfrak{T}^K_1 + h_K \mathfrak{T}^K_2 + \frac{2c_1 \tau}{h_K} \left(\mathfrak{T}^K_3 + \mathfrak{T}^K_4 \right)\right).
\end{align}
For the first term from~\eqref{eq:hdG_error_triangle_inequality}, we use the BDM approximation~\eqref{l:BDM_approximation} to get
\begin{equation}
\label{eq:hdG_error_first}
\mathfrak{T}_1^K \leq {c_2} h_K^{2k} \left|\boldsymbol{u}\right|_{H^{k+1}\left(K\right)}^2.
\end{equation}
Next we use the local trace inequality~\eqref{l:local_trace} to get
\begin{align}
\label{eq:hdG_error_second_local_trace}
\mathfrak{T}_2^K & \leq c_3 \left(\frac{1}{h_K}\left|\boldsymbol{u} - \boldsymbol{w_h}\right|_{H^1(K)}^2 + h_K\left|\boldsymbol{u} - \boldsymbol{w_h}\right|_{H^2(K)}^2\right).
\end{align}
Let $\mathcal{L}^k \boldsymbol{u}$ be the usual Lagrange interpolant of degree $k$ of $\boldsymbol{u}$ (see \cite[Example~1.31]{MR2050138}). Using the triangle inequality followed by the local inverse inequality~\eqref{l:local_inverse}, the local Lagrange approximation \cite[Example~1.106]{MR2050138} and~\eqref{l:BDM_approximation}, \eqref{eq:hdG_error_second_local_trace} becomes 
\begin{align}
\label{eq:hdG_error_second}
 \nonumber \mathfrak{T}_2^K & \leq c_3 \left(\frac{1}{h_K}\left|\boldsymbol{u} - \boldsymbol{w_h}\right|_{H^1(K)}^2 + 2h_K\left|\boldsymbol{u} - \mathcal{L}^k \boldsymbol{u}\right|_{H^2(K)}^2 + 2h_K\left|\mathcal{L}^k \boldsymbol{u} - \boldsymbol{w_h}\right|_{H^2(K)}^2\right) \\  \nonumber
& \leq c_3 \left(\left(c_4 + 2c_5 \right) h_K^{2k-1} \left|\boldsymbol{u}\right|_{H^{k+1}\left(K\right)}^2 + \frac{2 c_6}{h_K}\left|\mathcal{L}^k \boldsymbol{u} - \boldsymbol{w_h}\right|_{H^1(K)}^2\right) \\ \nonumber
& \leq c_3 \left(\left(c_4 + 2c_5 \right) h_K^{2k-1} \left|\boldsymbol{u}\right|_{H^{k+1}\left(K\right)}^2 + \frac{4 c_6}{h_K}\left|\mathcal{L}^k \boldsymbol{u} - \boldsymbol{u}\right|_{H^1(K)}^2 + \frac{4 c_6}{h_K}\left|\boldsymbol{u} - \boldsymbol{w_h}\right|_{H^1(K)}^2\right) \\
& \leq c_3(c_4 + 2c_5 + 4c_6 (c_7 + c_8)) h_K^{2k-1} \left|\boldsymbol{u}\right|_{H^{k+1}\left(K\right)}^2 .
\end{align}
For the third term in~\eqref{eq:hdG_error_triangle_inequality}, we use~\eqref{l:local_trace} and~\eqref{l:BDM_approximation}, to get
\begin{align}
\label{eq:hdG_error_third}
\mathfrak{T}_3^K \leq c_9 \left(\frac{1}{h_K}\left\|\boldsymbol{u} - \boldsymbol{w_h}\right\|_{K}^2 + h_K\left|\boldsymbol{u} - \boldsymbol{w_h}\right|_{H^1(K)}^2\right) \leq c_9 c_{10} h_K^{2k+1} \left|\boldsymbol{u}\right|_{H^{k+1}\left(K\right)}^2.
\end{align}
The last term in~\eqref{eq:hdG_error_triangle_inequality} is bounded using~\eqref{l:trace_approximation} as follows
\begin{equation}
\label{eq:hdG_error_fourth}
\mathfrak{T}_4^K \leq c_{11}
h_K^{2k+1} \left|\boldsymbol{u}\right|_{H^{k+1}\left(K\right)}^2.
\end{equation}
Finally, the local $L^2$-projection approximation~\eqref{l:L2_approximation} gives
\begin{equation}
\label{eq:hdG_error_pressure}
 \inf_{q \in Q_h^{k-1}} \left\|p- q_h\right\|_{\Omega} = \left\|p-\Psi_h^{k-1}(p)\right\|_{\Omega} \leq \tilde{c_1} h_K^{k} \|p\|_{H^{k}\left(\mathcal{T}_h\right)}.
\end{equation}
Thus, putting together~\eqref{eq:hdG_error_triangle_inequality} with~\eqref{eq:hdG_error_first}, \eqref{eq:hdG_error_second}, \eqref{eq:hdG_error_third}, \eqref{eq:hdG_error_fourth}, \eqref{eq:hdG_error_pressure} and shape regularity of the mesh we get
\begin{equation*}
\inf_{\left(\boldsymbol{v_h}, {\tilde{v}_h}, q_h\right) \in \boldsymbol{V_h}} |||\left(\boldsymbol{u}-\boldsymbol{v_h},\tilde{u}- {\tilde{v}_h}, p - q_h\right)|||_h \leq \hat{C} h^k \left( \sqrt{\nu} \|\boldsymbol{u}\|_{H^{k+1}\left(\mathcal{T}_h\right)} + \frac{1}{\sqrt\nu}\|p\|_{H^k\left(\mathcal{T}_h\right)} \right)
\end{equation*}
with
\begin{equation*}
 \hat{C} := \max \left\{\sqrt{c_2+ c_3(c_4 + 2c_5 + 4c_6 (c_7 + c_8))+ 2\tau c_1c_9 c_{10} + 2\tau c_1c_{11}}
 , \tilde{c}_1\right\},
\end{equation*}
and the result~\eqref{eq:hdG_error} follows from Lemma~\ref{l:TVNF_Cea}.
\end{proof}

\subsection{NVTF boundary conditions}
\label{sec:hdG_nvtf}

As we mentioned before, the analysis in case of NVTF boundary conditions~\eqref{eq:NVTF} is similar. Thus, we just highlight the main differences. So if we consider NVTF boundary conditions~\eqref{eq:NVTF}, then to discretise the velocity we use the following BDM space
\begin{align*}
 \boldsymbol{BDM_{h,0}^k} & := \left\{\boldsymbol{v_h} \in \boldsymbol{BDM_{h}^k}: \ {\left(\boldsymbol{v_h}\right)_n} = 0 \mbox{ on } \Gamma\right\}. 
\end{align*} 
For the Lagrange multiplier we use polynomial space $M_h^{k-1}$. And the pressure is discretised using
\begin{align*}
Q_{h,0}^{k-1} & := \left\{q_h \in Q_{h}^{k-1}: \ \int_{\Omega} q_h \dx = 0\right\} . 
\end{align*}
In this case our product space becomes $\boldsymbol{V_h} := \boldsymbol{BDM_{h,0}^k} \times M_h^{k-1}$ and we pose the following discrete problem.\\
\textit{Find $\left(\boldsymbol{u_h}, {\tilde{u}_h}, p_h\right) \in \boldsymbol{V_h} \times Q_{h,0}^{k-1}$ such that for all $\left(\boldsymbol{v_h}, {\tilde{v}_h}, q_h\right) \in \boldsymbol{V_h}\times Q_{h,0}^{k-1}$}
\begin{equation}
\label{eq:NVTF_variational_formulation}
 \left\{
		\begin{array}{rclcl}
   a \left(\left(\boldsymbol{u_h}, {\tilde{u}_h}\right),\left(\boldsymbol{v_h}, {\tilde{v}_h}\right)\right) & \ + & b\left(\left(\boldsymbol{v_h}, {\tilde{v}_h}\right), p_h\right) & = & \displaystyle\int_{\Omega}\boldsymbol{f}\boldsymbol{v_h} \dx + \int_{\Gamma} g \tilde{v}_h \ds \\
		 & & b\left(\left(\boldsymbol{u_h}, {\tilde{u}_h}\right), q_h\right) & = & 0
   \end{array}
 \right. .
\end{equation}
In obtaining~\eqref{eq:NVTF_variational_formulation} the only difference step in the derivation is that now~\eqref{eq:lagrange_multiplier} is replaced by
\begin{align*}
 -\int_{\Omega} \nabla \cdot \left(\nu \grad \boldsymbol{u} \right) \boldsymbol{v_h} \dx + \int_{\Omega} \grad p \cdot \boldsymbol{v_h} \dx &= \sum_{K \in \mathcal{T}_h} \left(\int_K \nu \grad \boldsymbol{u} : \grad \boldsymbol{v_h} \dx - \int_K p \nabla \cdot \boldsymbol{v_h} \dx  \right. \\ \nonumber
& \left.  - \int_{\partial K} \sigma_{nt}  \left(\left(\boldsymbol{v_h}\right)_t - \tilde{v}_h\right) \ds\right) - \int_{\Gamma} g  \tilde{v}_h \ds.
\end{align*}
Concerning the analysis, the proofs of all the results presented in the last sections remain essentially unchanged.


\section{The domain decomposition preconditioner}
\label{sec:dd_one_level}

Let us assume that we have to solve the following linear system
$$
\mathbf{A} \boldsymbol{U} = \boldsymbol{F}
$$
where $\mathbf{A}$ is the matrix arising from discretisation of the Stokes
equations on the domain $\Omega$, $\boldsymbol{U}$ is the vector of
unknowns and $\boldsymbol{F}$ is the right hand side. To accelerate the
performance of an iterative Krylov method applied to this system we
will consider domain decomposition preconditioners which are naturally
parallel \cite[Chapter~3]{MR3450068}. They are based on an overlapping partition of the
computational domain.

Let $\{\mathcal{T}_{h,i}\}_{i=1}^N$ be a partition of the
triangulation $\mathcal{T}_h$. For an integer value $l \geq 0$, we
define an overlapping decomposition
$\{\mathcal{T}_{h,i}^{l}\}_{i=1}^N$ such that $\mathcal{T}_{h,i}^{l}$
is a set of all triangles from $\mathcal{T}_{h,i}^{l-1}$ and all
triangles from $\mathcal{T}_h \setminus \mathcal{T}_{h,i}^{l-1}$ that
have non-empty intersection with $\mathcal{T}_{h,i}^{l-1}$, and
$\mathcal{T}_{h,i}^0 = \mathcal{T}_{h,i}$. With this definition the width of
the overlap will be of $2l$. Furthermore, if $W_h$
stands for the finite element space associated to $\mathcal{T}_h$,
$W_{h,i}^{l}$ is the local finite element spaces on
$\mathcal{T}_{h,i}^{l}$ that is a triangulation of $\Omega_i$. 

Let $\mathcal{N}$ be the set of indices of degrees of
freedom of $W_h$ and $\mathcal{N}_i^{l}$ the set of indices of degrees of freedom of $W_{h,i}^{l}$ for $l \geq 0$. Moreover, we define the restriction operator $\mathbf{R_i}: W_h \rightarrow W_{h,i}^{l}$ as a rectangular matrix $|\mathcal{N}_i^{l}| \times |\mathcal{N}|$ such that if $\boldsymbol{V}$ is the vector of degrees of freedom of $v_h \in W_h$, then $\mathbf{R_i} \boldsymbol{V}$ is the vector of degrees of freedom of $W_{h,i}^{l}$ in $\Omega_i$. Abusing notation we denote by $\mathbf{R_i}$ both the operator, and its associated matrix. The extension operator from $W_{h,i}^{l}$ to $W_h$ and its associated matrix are both given by $\mathbf{R_i}^T$. In addition we introduce a partition of unity $\mathbf{D_i}$ as a diagonal matrix $|\mathcal{N}_i^{l}| \times |\mathcal{N}_i^{l}|$ such that
\begin{equation}
\mathbf{Id} = \sum_{i=1}^N \mathbf{R_i}^T \mathbf{D_i} \mathbf{R_i},
\end{equation}
where $\mathbf{Id} \in \mathbb{R}^{|\mathcal{N}| \times |\mathcal{N}|}$ is the identity matrix.

We are ready to present the first preconditioner, called Restricted Additive Schwarz (RAS) \cite{MR1718707} 
, given by
\begin{equation}
\label{eq:dd_RAS}
\mathbf{M_{RAS}}^{-1} = \sum_{i=1}^N \mathbf{R_i}^T \mathbf{D_i} (\mathbf{R_i} \mathbf{A} \mathbf{R_i}^T)^{-1} \mathbf{R_i}.
\end{equation}
We also introduce a new preconditioner that is a modification of the above one. The modification is similar to the Optimized RAS \cite{MR2357620}, however we do not use Robin IC. For this, let $\mathbf{B_i}$ be the matrix associated to a discretisation of~\eqref{eq:stokes} in $\Omega_i$ where we impose either TVNF~\eqref{eq:TVNF} or NVTF~\eqref{eq:NVTF} boundary conditions in $\Omega_i$. Then, the preconditioner reads
\begin{equation}
\label{eq:dd_MRAS}
\mathbf{M_{MRAS}}^{-1} = \sum_{i=1}^N \mathbf{R_i}^T \mathbf{D_i B_i}^{-1} \mathbf{R_i}.
\end{equation}

\begin{remark}
The improvement of convergence in the case of Optimized RAS depends on the choice of the parameter. This parameter is depending on the problem and discretisation. The big advantage of MRAS preconditioners is that they are parameter-free.
\end{remark}


\subsection{Partion of unity}
\label{sec:dd_partition}

The above definitions of the preconditioners can be associated with any discretisation of the problem. However, each discretisation involves the construction of a relevant partition of unity $\mathbf{D_i}$, $i=1,\dots,N$. We
discuss here the construction of $\mathbf{D_i}$ when the problem~\eqref{eq:stokes} is discretised by the hdG method in case $k=1$, either with TVNF boundary conditions~\eqref{eq:TVNF_variational_formulation}, or NVTF boundary conditions~\eqref{eq:NVTF_variational_formulation}. Let us introduce the piecewise linear functions $\tilde{\chi}_i^{l}$ of $\mathcal{T}_{h}$ such that
\begin{equation*}
	\tilde{\chi}_i^{l} = \left\{
\begin{array}{l l}
	1 & \mbox{on all nodes of } \mathcal{T}_{h,i}^0, \\
	0 & \mbox{on other nodes.}
\end{array}
\right. 
\end{equation*}
Now we define the piecewise linear functions $\chi_i^{l}$ of $\mathcal{T}_{h,i}^{l}$ as follows
\begin{equation*}
	\chi_i^{l} := \frac{\tilde{\chi}_i^l}{\sum_{j=1}^{N} \tilde{\chi}_j^l}.
\end{equation*}
Obviously $\sum_{i=1}^{N} \chi_i^l = 1$. We define the partition of
unity matrix $\mathbf{D_i}$ as a block diagonal matrix where first block $\mathbf{D_i^{BDM}}$ is associated with $\boldsymbol{BDM_h^1}$, second $\mathbf{D_i^{M}}$ with $M_h^{0}$ and third $\mathbf{D_i^{Q}}$ with $Q_h^{0}$.
The degrees of freedom of the BDM elements are associated with the normal components on the edges of the mesh. For these finite elements, the diagonal of $\mathbf{D_i^{BDM}}$ is a vector obtained by interpolating $\chi_i^{l}$ at the two points of the edges. The degrees of freedom of the Lagrange multiplier finite elements are associated with the edges of the mesh. For these finite elements, the diagonal of $\mathbf{D_i^{M}}$ is a vector obtained by interpolating $\chi_i^{l}$ at the midpoints of the edges. For pressure finite elements, the diagonal of $\mathbf{D_i^{Q}}$ is a vector obtained by interpolating $\chi_i^{l}$ at the midpoints of the elements.


\section{Numerical results}
\label{sec:numerics}

In this section we present a series of numerical experiments aimed at confirming the theory developed in Section~\ref{sec:hdg}, and to give a computational comparison of the preconditioners discussed in the previous section. All experiments have been made by using FreeFem++~\cite{MR3043640}, which is a free software specialised in variational discretisations of partial differential equations.

\subsection{Convergence validation}
\label{sec:hdG_numerics}

The computational domain for both test cases considered here is the unit square $\Omega = \left(0,1\right)^2$. We present the results for
$k=1$, this is, the discrete space is given by $\boldsymbol{BDM_h^1} \times M_{h,0}^0 \times Q_h^0$ for TVNF boundary conditions and $\boldsymbol{BDM_{h,0}^1} \times M_h^0 \times Q_{h,0}^0$ for NVTF boundary conditions.  We test both the symmetric method
($\varepsilon=-1$) and the non-symmetric method ($\varepsilon=1$). For
both cases we have followed the recommendation given in~\cite[Section~2.5.2]{Lehrenfeldhesis} and taken $\tau =6$.

 The first example aims at verifying the formulation with
TVNF boundary conditions~\eqref{eq:TVNF_variational_formulation}. We
choose the right hand side $\boldsymbol{f}$ and the boundary datum $g$ such that the exact solution is given by
\begin{align*}
\boldsymbol{u} = \mbox{curl} \left[100\left(1-\cos((1-x)^2)\right)\sin(x^2)\sin(y^2)\left(1-\cos((1-y)^2)\right)\right], &&
p = \tan(xy).
\end{align*}
In Figures~\ref{fig:Trigon_sym} and \ref{fig:Trigon_non} we show the
results of the usual convergence order
tests for the symmetric case and the
non-symmetric case by plotting in log-log scale the error as a function of
the size of the mesh. We notice that they validate the theory from Section~\ref{sec:TVNF_error}. In addition, an optimal $h^2$ convergence rate is observed for $\|\boldsymbol{u}-\boldsymbol{u_h}\|_{\Omega}$. The proof of this fact is lacking, but it does not seem to be an easy task due to the nature of the boundary condition of problem~\eqref{eq:stokes}.

\begin{figure}[!ht]
\centering
    \subfloat[Symmetric bilinear form ($\varepsilon = -1$)\label{fig:Trigon_sym}]{%
      \includegraphics[width=0.47\textwidth]{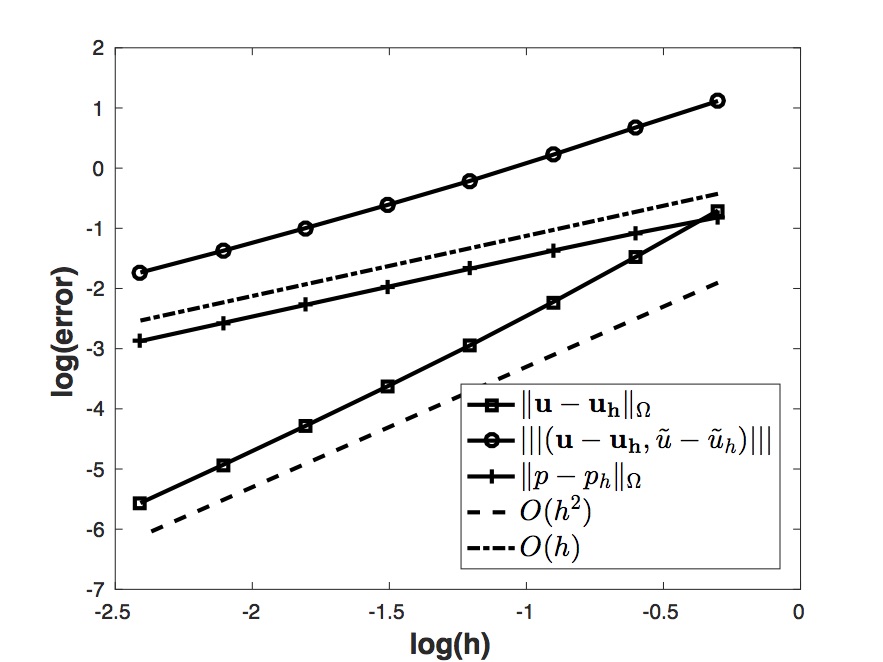}
    }
    \subfloat[Non-symmetric bilinear form ($\varepsilon = 1$)\label{fig:Trigon_non}]{%
      \includegraphics[width=0.47\textwidth]{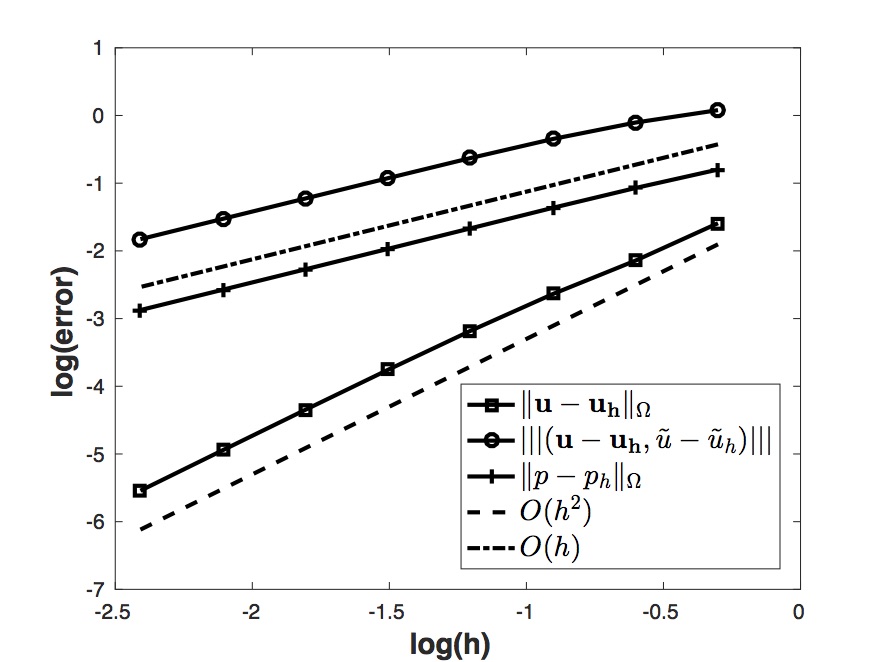}
    }
  \caption{Error convergence of the hdG method with TVNF boundary condition - the first example}
  \end{figure}
 The second example aims at verifying the formulation with
NVTF boundary conditions \eqref{eq:NVTF_variational_formulation}. We
choose the right hand side $\boldsymbol{f}$ and the boundary datum $g$ such that the exact solution is given by
\begin{align*}
\boldsymbol{u} = \mbox{curl} \left[x^2\left(1-x\right)^2y^2\left(1-y\right)^2\right], &&
p = x - y.
\end{align*}
In Figures~\ref{fig:bubble_sym} and \ref{fig:bubble_non} we show the
results of the usual convergence order
tests for the symmetric case and the
non-symmetric case by plotting in log-log scale the error as a function of
the size of the mesh. We notice that they validate the theory from Section~\ref{sec:TVNF_error}. And again, an optimal $h^2$ convergence rate is observed for $\|\boldsymbol{u}-\boldsymbol{u_h}\|_{\Omega}$.
\begin{figure}[!ht]
\centering
    \subfloat[Symmetric bilinear form ($\varepsilon = -1$)\label{fig:bubble_sym}]{%
      \includegraphics[width=0.47\textwidth]{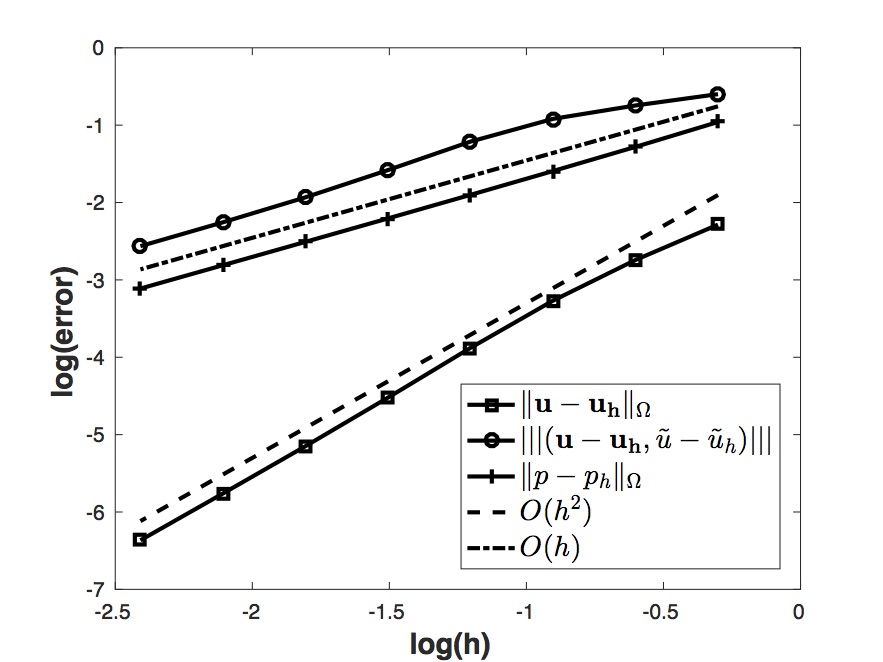}
    }
    \subfloat[Non-symmetric bilinear form ($\varepsilon = 1$)\label{fig:bubble_non}]{%
      \includegraphics[width=0.47\textwidth]{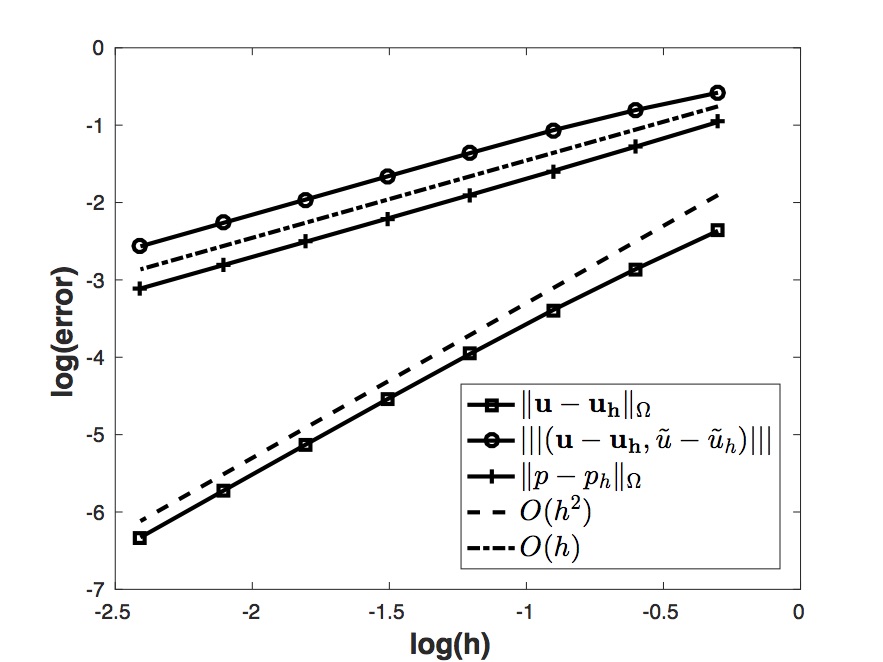}
    }
\caption{Error convergence of the hdG method with NVTF boundary condition - the second example}
  \end{figure}


\subsection{Comparison of different domain decomposition preconditioners}
\label{sec:dd_numerics}

In this section we compare the standard RAS preconditoner~\eqref{eq:dd_RAS} with the newly
introduced preconditioners, that is the ones based on non standard
IC. We call them MRAS preconditioners~\eqref{eq:dd_MRAS} and more precisely
TVNF-MRAS for which $\mathbf{B_i}$ is the matrix arising from the discretisation of~\eqref{eq:stokes} in $\Omega_i$ with IC~\eqref{eq:TVNF} on
$\partial \Omega_i$, and NVTF-MRAS for which $\mathbf{B_i}$ is the matrix arising from the discretisation of \eqref{eq:stokes} in $\Omega_i$ with IC~\eqref{eq:NVTF} on $\partial \Omega_i$.  As we mentioned before, our preconditioners do not depend on the used discretisation, that is why we add also similar preconditioners but based
on a more standard discretisation, that is, the lowest order Taylor-Hood
discretisation~\cite[Chapter II, Section 4.2]{MR851383}. In all cases,
they are used in conjunction with a Krylov iterative solver such as GMRES~\cite{MR848568}. In addition, $N$ stands for the number of subdomains in all tables.  In all tables we present the number of iterations needed to
achieve an euclidean norm of the error (with respect to the one domain
solution) smaller than $10^{-6}$. We have implemented the RAS
preconditioner~\eqref{eq:dd_RAS} and the MRAS~\eqref{eq:dd_MRAS}, using both TVNF and NVTF interface conditions. 

We start with the second example from the previous
section. However, now we consider the symmetric
($\varepsilon=-1$) formulation with TVNF boundary
conditions~\eqref{eq:TVNF_variational_formulation}. The mesh is
uniform and contains 125 000 triangles for a total of 565 003 degrees
of freedom for the Taylor-Hood discretisation and 689 000 degrees of
freedom for the hdG discretisation. We use a random initial guess for
the GMRES iterative solver. The overlapping decomposition into subdomains can be
uniform (Unif) or generated by METIS (MTS) and it has two layers of mesh size $h$
in the overlap. 
\begin{table}[!ht]
\begin{center}
\begin{adjustbox}{max width=\textwidth}
  \begin{tabular}{c | c c | c c | c c || c c | c c | c c }
    & \multicolumn{6}{ c ||}{\textbf{Taylor-Hood}} & \multicolumn{6}{ c }{\textbf{hdG}} \\
    N & \multicolumn{2}{c |}{RAS} & \multicolumn{2}{c |}{NVTF-MRAS} & \multicolumn{2}{c ||}{TVNF-MRAS} & \multicolumn{2}{c |}{RAS} & \multicolumn{2}{c |}{NVTF-MRAS} & \multicolumn{2}{c}{TVNF-MRAS}  \\
    & Unif & MTS & Unif & MTS & Unif & MTS & Unif & MTS & Unif & MTS & Unif & MTS   \\
    \hline
		\textbf{4} & 133 & 311 & 40 & 39 & 37 & 37 & 58 & 95 & 41 & 45 & 53 & 50  \\
		\textbf{9} & 336 & 563 & 58 & 58 & 52 & 60 & 94 & 131 & 62 & 66 & 69 & 81  \\
		\textbf{16} & 315 & 691 & 60 & 76 & 59 & 73 & 101 & 151 & 68 & 85 & 80 & 100  \\
		\textbf{25} & 427 & 774 & 76 & 93 & 71 & 90 & 127 & 186 & 77 & 100 & 103  & 119  \\
		\textbf{64} & 630 & 1132 & 113 & 147 & 112 & 132 & 196 & 280 & 126 & 172 & 148 & 183  \\
		\textbf{100} & 769 & 1246 & 136 & 174 & 132 & 169 & 247 & 348 & 151 & 205 & 175 & 228  \\
		\textbf{144} & 929 & 1434 & 158 & 201 & 155 & 192 & 306 & 408 & 178 & 228 & 192 & 259  \\
		\textbf{196} & 1000 & 1637 & 180 & 239 & 168 & 224 & 354 & 480 & 198 & 326 & 212 & 299  \\
		\textbf{256} & 1133 & 1805 & 201 & 265 & 183 & 286 & 403 & 536 & 226 & 358 & 233 & 341  \\
  \end{tabular}
\end{adjustbox}
\end{center}
\caption{Preconditioners comparison - the first test case}
\label{tab:Bubble}
\end{table}

The first thing that we can notice from Table~\ref{tab:Bubble} is the important convergence
improvement in case of RAS applied to a~system resulting from a hdG
discretisation in comparison to the RAS applied to the system
resulting from the Taylor-Hood discretisation despite the fact that
the number of degrees of freedom is slightly bigger in the first case. The change in discretisation
presumably leads to better conditioned systems to solve. Also the MRAS
preconditioner with both discretisations perform better than the
standard RAS method which fully justifies the use of the new
IC no matter the discretisation method. Moreover, as expected,
the number of iterations increases with respect to the number of the
subdomains and this behaviour is common to the three preconditioners. It is worth noticing that this increase is slower than the expected linear one.

We also plot the convergence of the error for the different
discretisations in Figure~\ref{fig:Bubble_precontioner_uniform}
and~\ref{fig:Bubble_precontioner_metis}. We observe that in all cases
the MRAS preconditioner~\eqref{eq:dd_MRAS} shortens the
plateau region in the convergence curves significantly which leads,
automatically, to an important reduction in the number of iterations.
\begin{figure}[!ht]
\centering
    \subfloat[Taylor-Hood]{\includegraphics[width=0.47\textwidth]{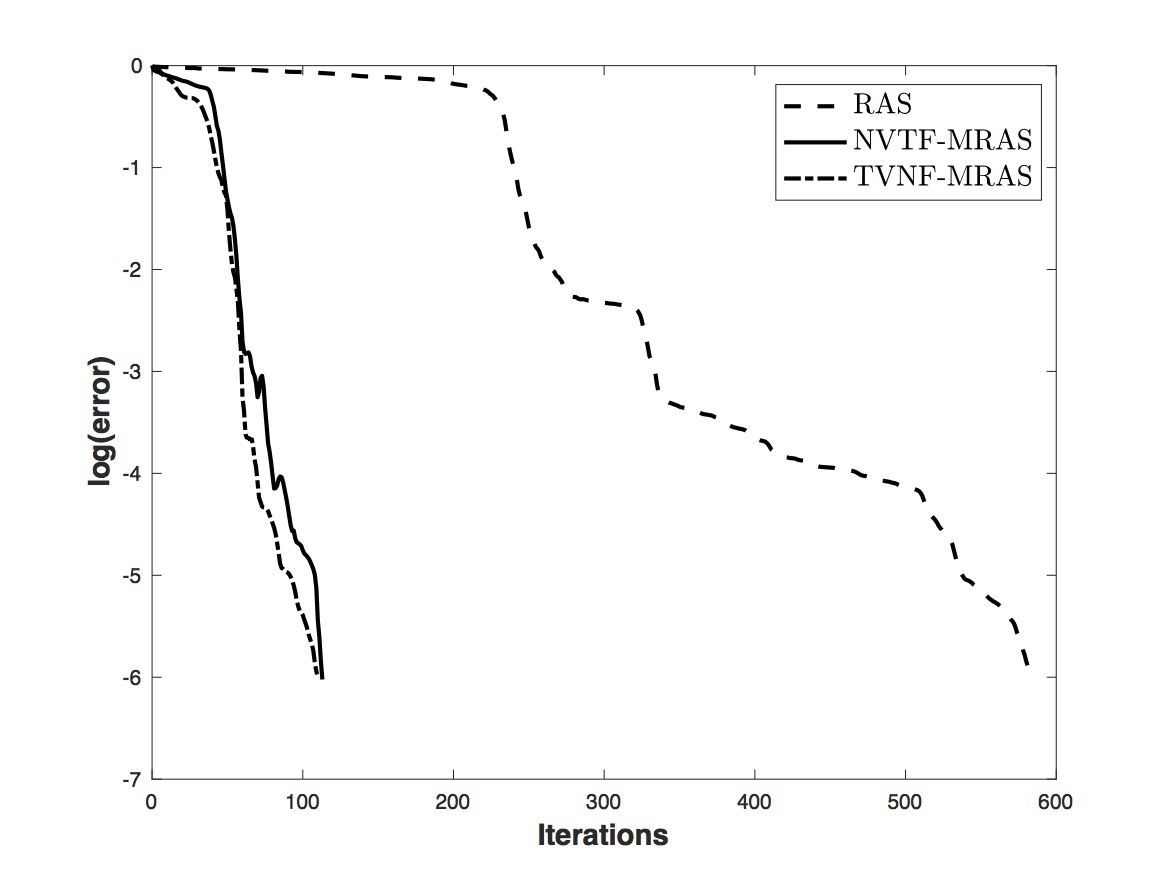}}
    \subfloat[hdG]{\includegraphics[width=0.47\textwidth]{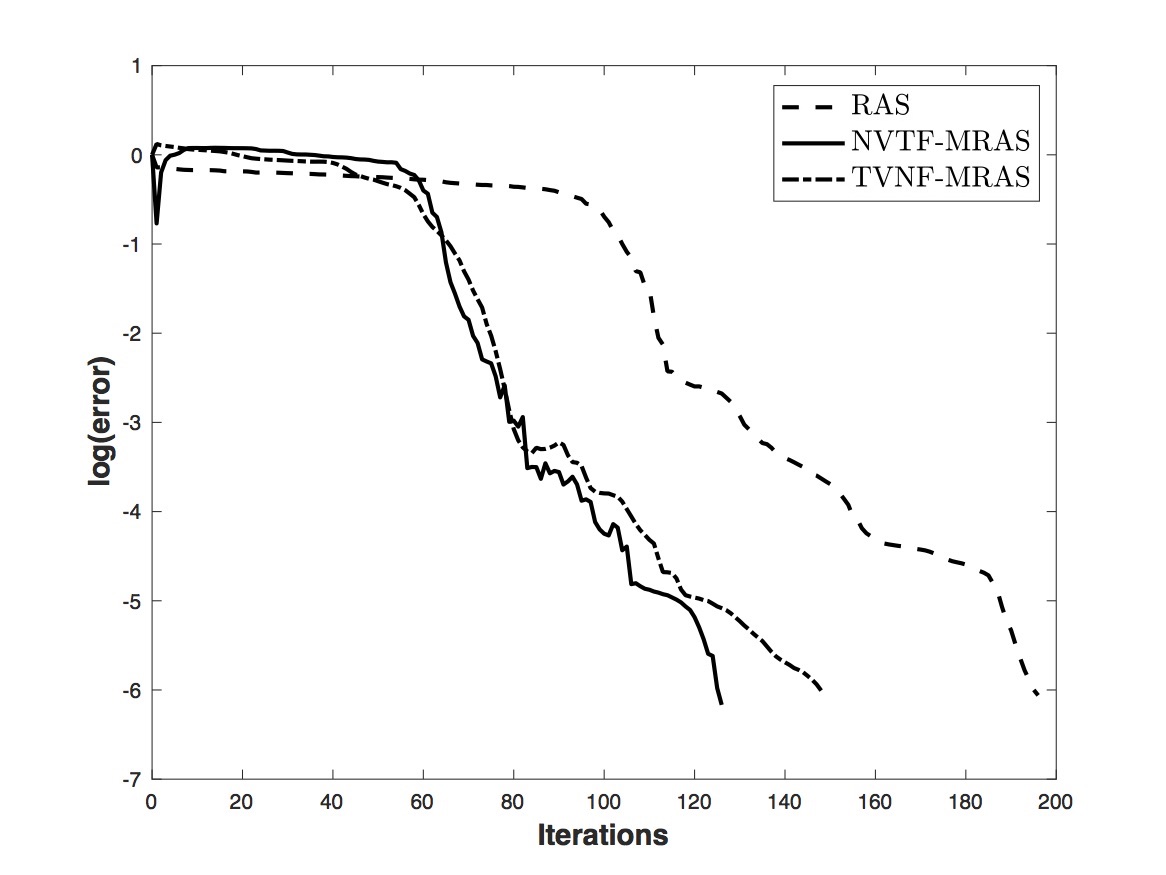}}
		\caption{Convergence of error for uniform decomposition in the $8 \times 8$  subdomains case - the first test case}
  \label{fig:Bubble_precontioner_uniform}
  \end{figure}
	
\begin{figure}[!ht]
\centering
    \subfloat[Taylor-Hood]{%
      \includegraphics[width=0.47\textwidth]{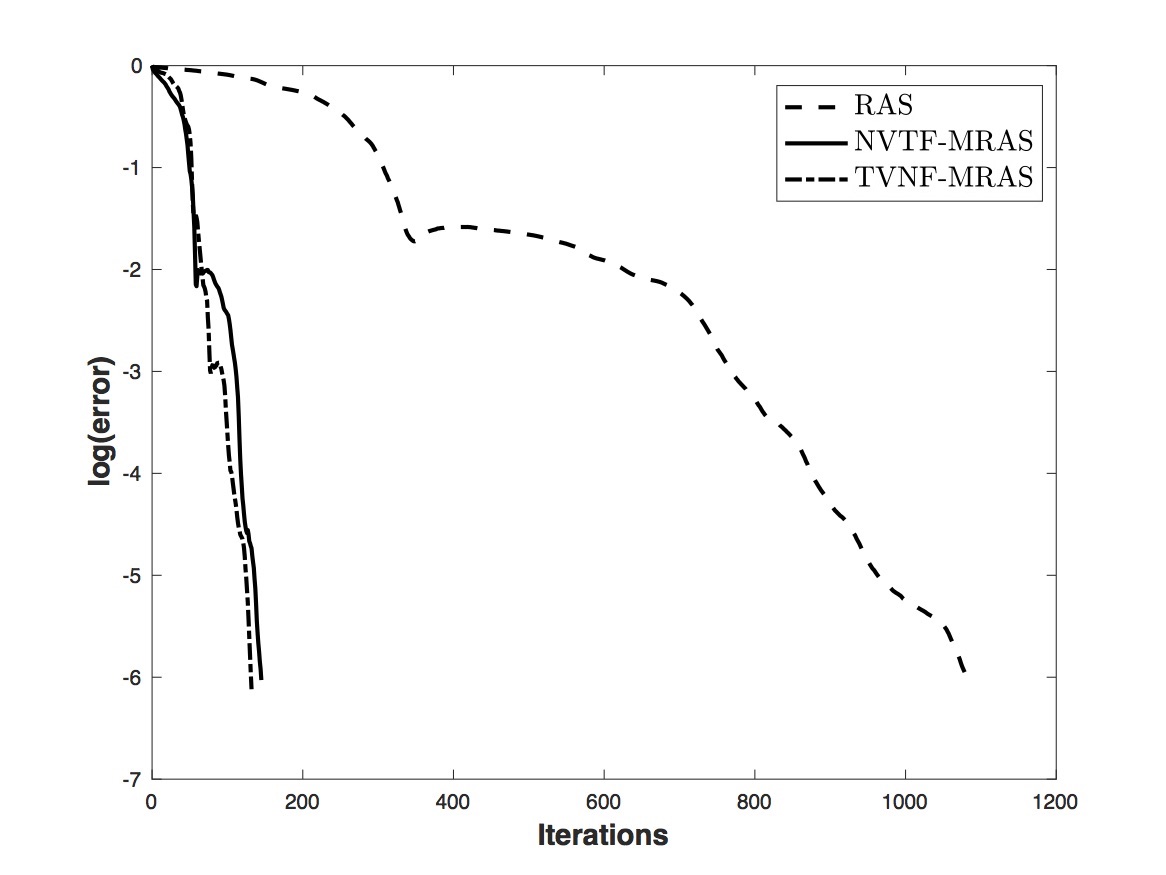}
    }
    \subfloat[hdG]{%
      \includegraphics[width=0.47\textwidth]{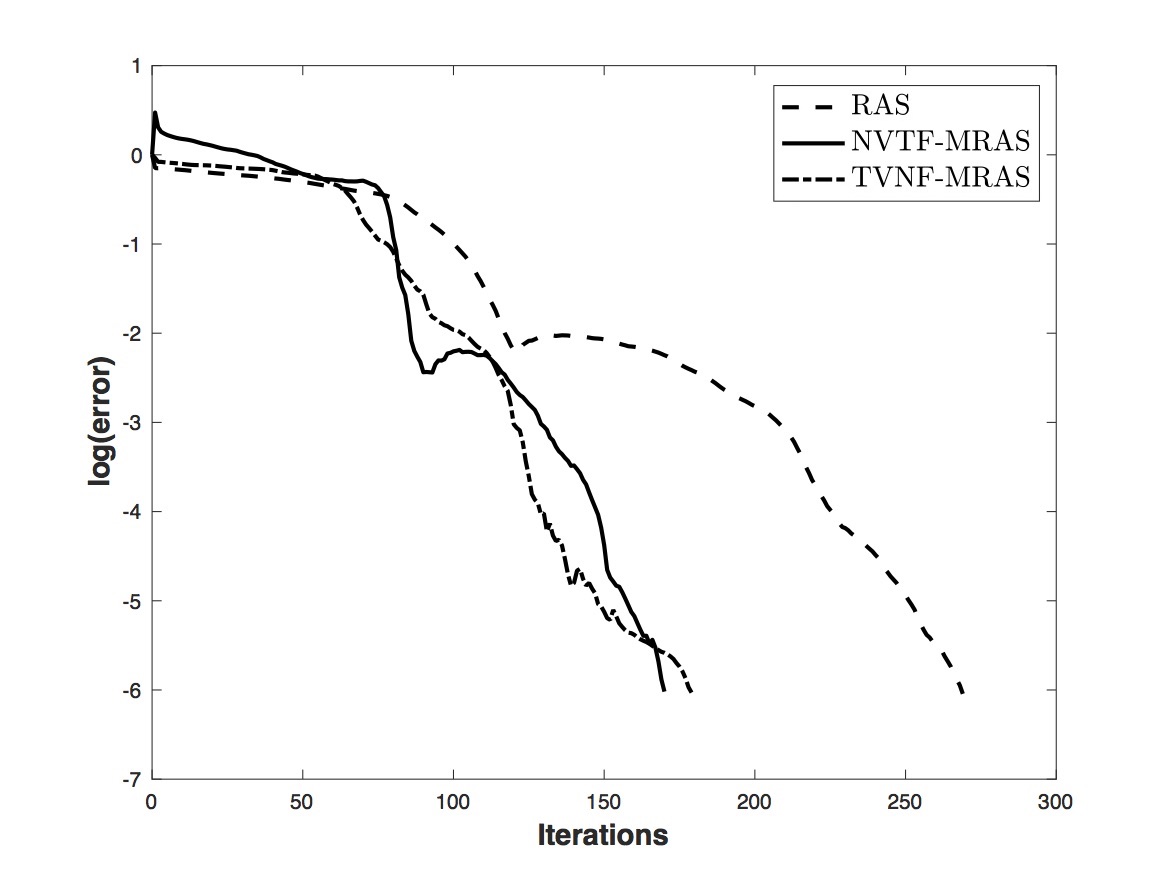}
    }
	\caption{Convergence of error for METIS decomposition in the 64 subdomains case - the first test case}
 \label{fig:Bubble_precontioner_metis}
  \end{figure}
     \newpage
Now we consider the Poiseuille problem and we choose the right hand side $\boldsymbol{f}$ and the TVNF boundary condition such that the exact solution is given by
\begin{align*}
\boldsymbol{u} = [4y(1-y),0]^T, &&
p = 4-8x.
\end{align*}
The mesh is again uniform and contains 125 000 triangles for a total of 565 003 degrees
of freedom for the Taylor-Hood discretisation and 689 000 degrees of
freedom for the hdG discretisation. We use a random initial guess for
the GMRES iterative solver. The overlapping decomposition into subdomains can be
uniform (Unif) or generated by METIS (MTS) and it has three layers of mesh size $h$
in the overlap. 
\begin{table}[!ht]
\begin{center}
\begin{adjustbox}{max width=\textwidth}
  \begin{tabular}{c | c c | c c | c c || c c | c c | c c }
    & \multicolumn{6}{ c ||}{\textbf{Taylor-Hood}} & \multicolumn{6}{ c }{\textbf{hdG}} \\
    N & \multicolumn{2}{c |}{RAS} & \multicolumn{2}{c |}{NVTF-MRAS} & \multicolumn{2}{c ||}{TVNF-MRAS} & \multicolumn{2}{c |}{RAS} & \multicolumn{2}{c |}{NVTF-MRAS} & \multicolumn{2}{c}{TVNF-MRAS}  \\
    & Unif & MTS & Unif & MTS & Unif & MTS & Unif & MTS & Unif & MTS & Unif & MTS   \\
    \hline
		\textbf{4} & 117 & 220 & 36 & 39 & 38 & 36 & 58 & 95 & 39 & 47 & 54 & 48  \\
		\textbf{9} & 294 & 421 & 63 & 60 & 54 & 54  & 103 & 129 & 66 & 67 & 77 & 78  \\
		\textbf{16} & 236 & 510 & 59 & 73 & 61 & 68 & 98 & 153 & 65 & 83 & 74 & 94  \\
		\textbf{25} & 300 & 642 & 68 & 89 & 72 & 83 & 120 & 184 & 77 & 103 & 88 & 115  \\
		\textbf{64} & 454 & 916 & 102 & 144 & 100 & 122 & 188 & 279 & 117 & 160 & 120 & 165  \\
		\textbf{100} & 559 & 1088 & 122 & 173 & 116 & 154 & 225 & 349 & 140 & 198 & 138 & 215  \\
		\textbf{144} & 940 & 1251 & 176 & 195 & 145 & 215 & 342 & 395 & 198 & 231 & 183 & 232  \\
		\textbf{196} & 781 & 1346 & 166 & 230 & 146 & 242 & 325 & 486 & 191 & 277 & 173 & 284  \\
		\textbf{256} & 881 & 1553 & 189 & 269 & 159 & 272 & 368 & 538 & 210 & 316 & 195 & 309 
  \end{tabular}
\end{adjustbox}
\end{center}
\caption{Preconditioners comparison - the Poiseuille problem}
\label{tab:Poiseuille}
\end{table}

The conclusions stay the same as in previous example since the reusults form Table~\ref{tab:Poiseuille} are similar to the previous ones. We consider a different problem, however on the same mesh. Hence the global matrix is the same in both cases. Thus, we can notice a reduction  in the number of iterations caused by the increase of the width of the overlap.

We also plot the convergence of the error for the different
discretisations in Figure~\ref{fig:Poiseuille_precontioner_uniform}
and~\ref{fig:Poiseuille_precontioner_metis}. We observe that in all cases once again
the MRAS preconditioner~\eqref{eq:dd_MRAS} shortens the
plateau region in the convergence curves significantly which leads,
automatically, to an important reduction in the number of iterations.
\begin{figure}[!ht]
\centering
    \subfloat[Taylor-Hood]{\includegraphics[width=0.47\textwidth]{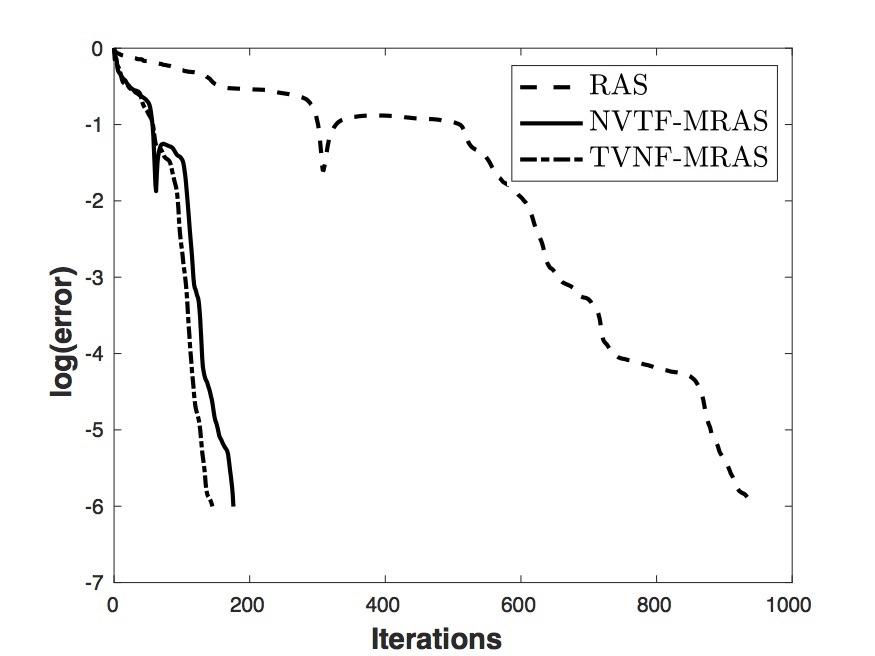}}
    \subfloat[hdG]{\includegraphics[width=0.47\textwidth]{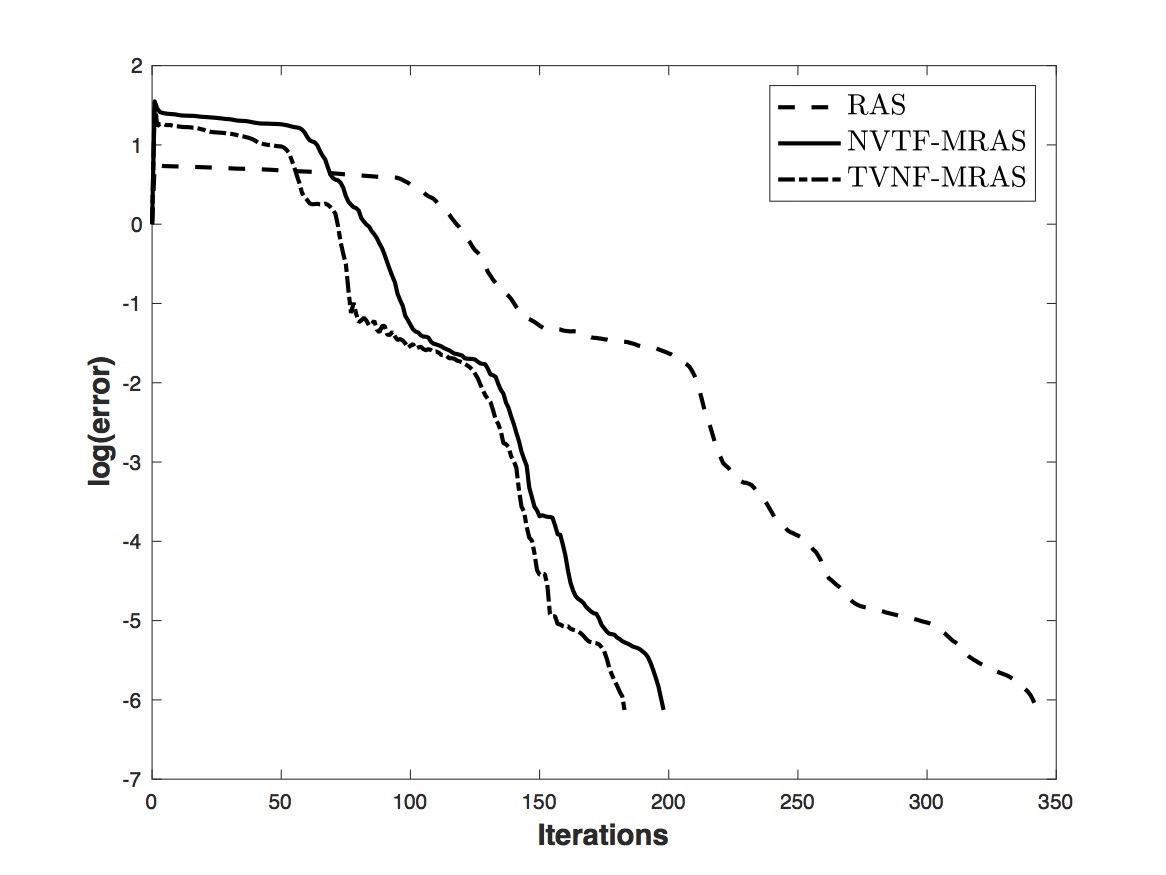}}
		\caption{Convergence of error for uniform decomposition in the $12 \times 12$ subdomains case - the Poiseuille problem}
  \label{fig:Poiseuille_precontioner_uniform}
  \end{figure}

\begin{figure}[!ht]
\centering
    \subfloat[Taylor-Hood]{%
      \includegraphics[width=0.47\textwidth]{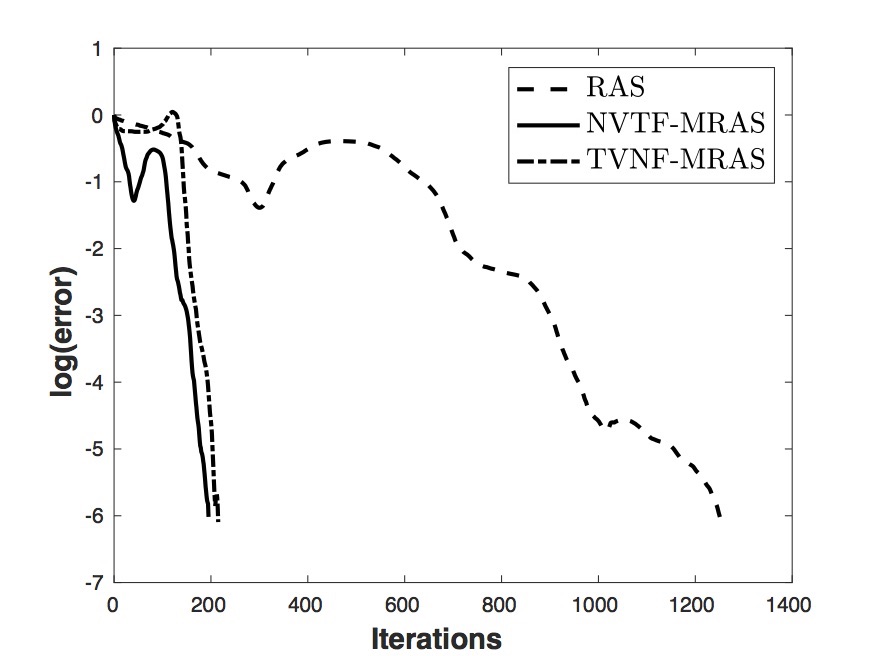}
    }
    \subfloat[hdG]{%
      \includegraphics[width=0.47\textwidth]{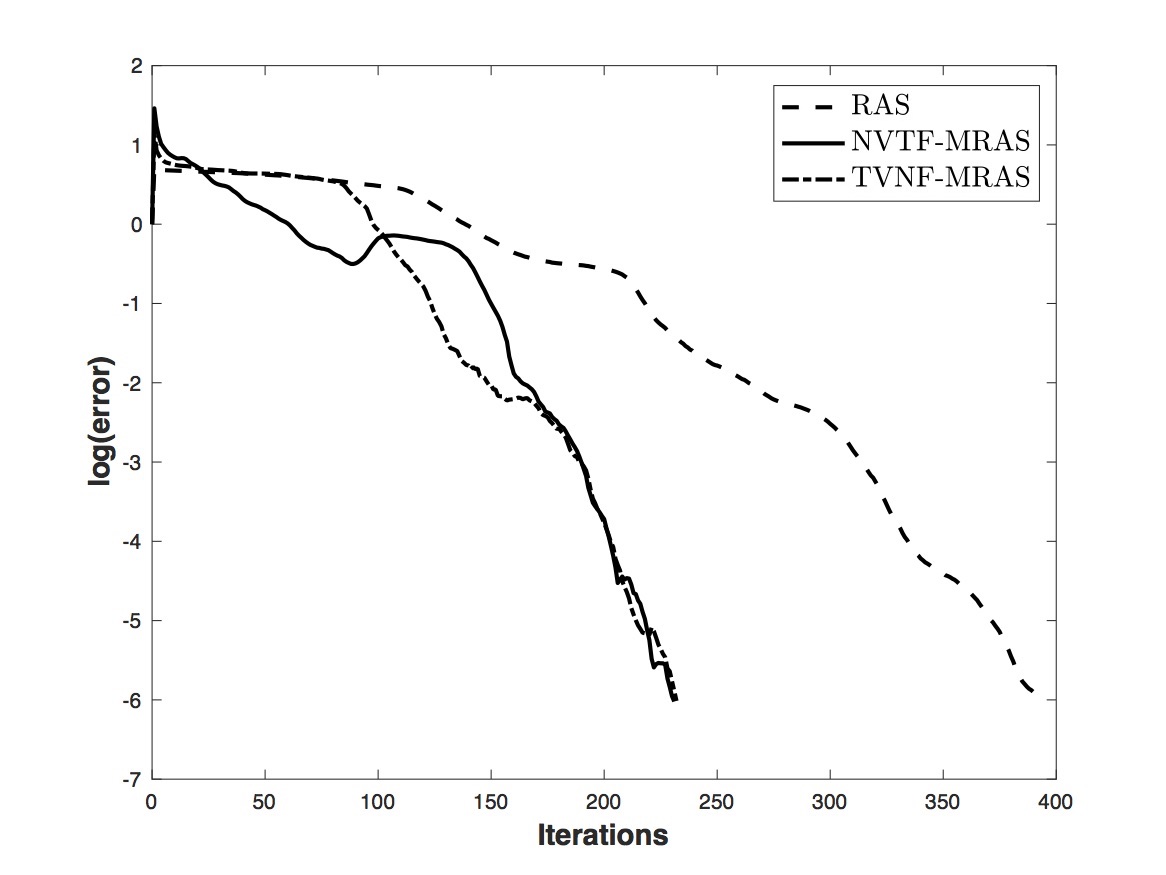}
    }
	\caption{Convergence of error for METIS decomposition in the 144 subdomains case - the Poiseuille problem}
 \label{fig:Poiseuille_precontioner_metis}
  \end{figure}
     \newpage
 The last example is on a T-shaped domain $\Omega = \mbox{Int}\left([0,1.5] \times [0,1] \cup [0.5,1] \times [-1,0]\right)$, and we impose mixed boundary conditions given by
\begin{equation*}
 \left\{
  \begin{array}{c l}
  	\boldsymbol{u}(x,y) = (4y(1-y),0)^T & \mbox{if } x=0 \\
	\sigma_{nn}(x,y) = 0, \quad u_t(x,y) = 0& \mbox{if } x=1.5 \\
	\boldsymbol{u}(x,y) = (0,0)^T & \mbox{otherwise}
  \end{array}
\right. . 
\end{equation*}
\begin{figure}[!ht]
\centering
    \subfloat[Velocity field $\boldsymbol{u}$]{%
      \includegraphics[width=0.49\textwidth]{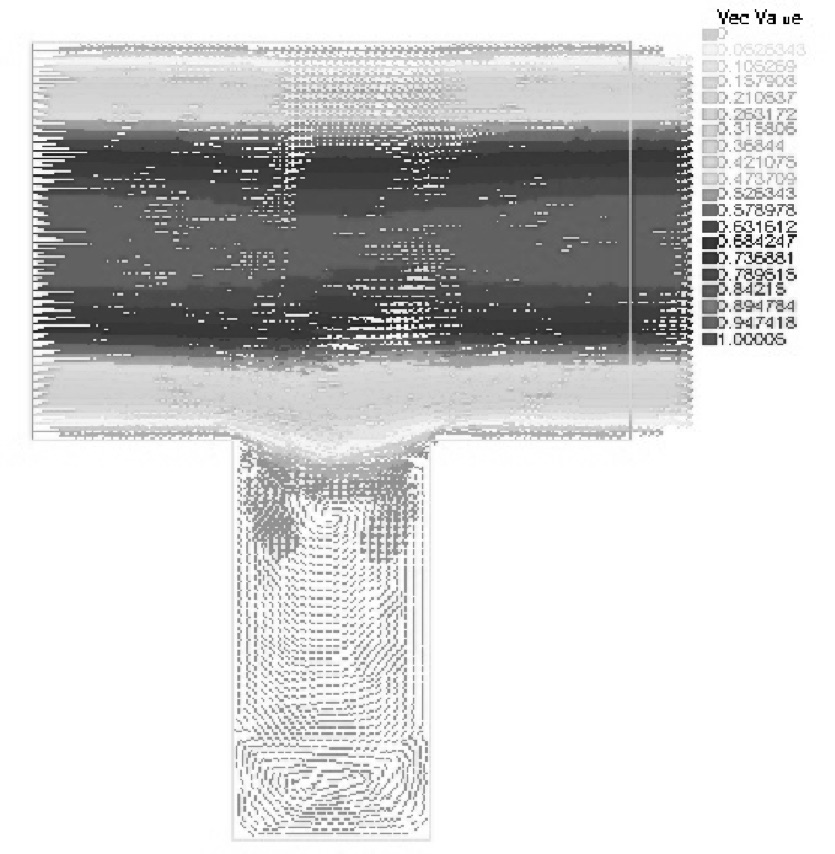}
    }\hfill
    \subfloat[Pressure $p$]{%
      \includegraphics[width=0.49\textwidth]{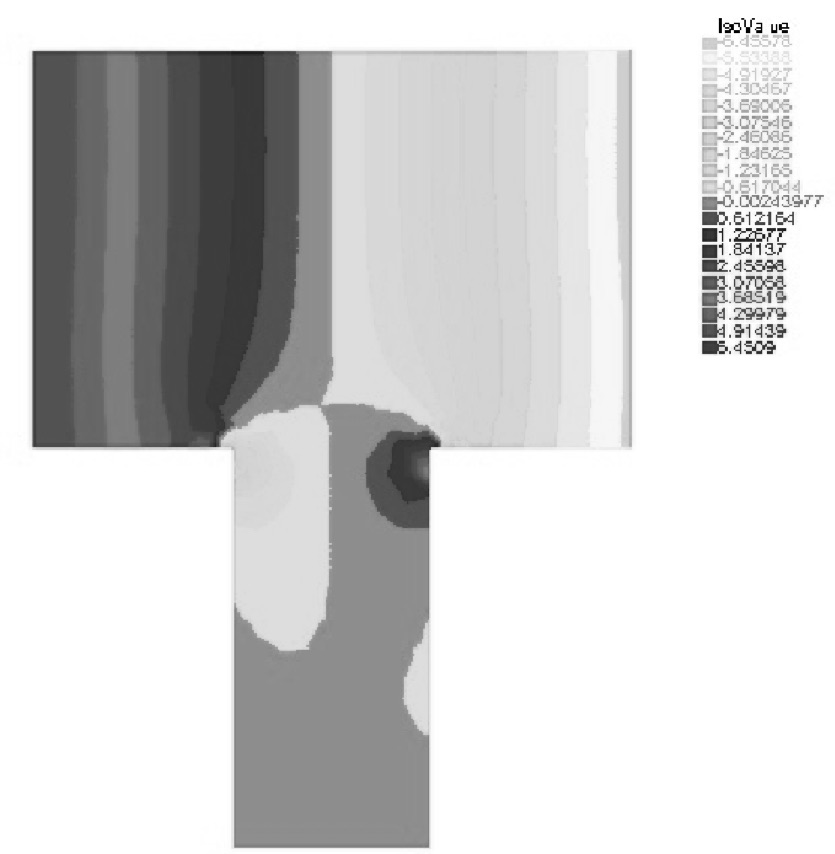}
    }
  \caption{Numerical solution of the T-shaped domain problem}
  \label{fig:Poiseuille_T_shape}
  \end{figure}
In Figure~\ref{fig:Poiseuille_T_shape} we plot the numerical solution
obtained with the hdG discretisation using $\tau = 6$ on a coarse mesh. In this case,
we used a mesh containing 379 402 triangles, which gives linear
systems of a size 1 712 352 for the Taylor-Hood
discretisation and 2 089 735 for the hdG
discretisation. The initial guess in the GMRES iterative solver is
zero. The overlapping decomposition into subdomains is generated by METIS and it has two layers of mesh size $h$
in the overlap. 
\begin{table}[!ht]
\begin{center}
  \begin{tabular}{c | c | c | c || c | c | c }
    & \multicolumn{3}{ c ||}{\textbf{Taylor-Hood}} & \multicolumn{3}{ c }{\textbf{hdG}} \\
    N & {RAS} & {NVTF-MRAS} & {TVNF-MRAS} & {RAS} & {NVTF-MRAS} & {TVNF-MRAS}  \\
   \hline
		\textbf{50} & 752 & 121 & 105 & 209 & 132 & 135 \\
		\textbf{100} & 903 & 175 & 147 & 307 & 190 & 197 \\
		\textbf{200} & 1272 & 245 & 211 & 441 & 264 & 281 \\
		\textbf{400} & 1747 & 341 & 342 & 613 & 366 & 399 \\
		\textbf{800} & 2433 & 469 & 417 & 863 & 650 & 549 
  \end{tabular}
\end{center}
\caption{Preconditioners comparison - the T-shaped domain problem}
\label{tab:T_shape}
\end{table}

According to Table~\ref{tab:T_shape} the conclusions remain the same, that is the standard RAS method performs far better when
applied to a hdG discretisation with respect to a Taylor-Hood one and
the MRAS preconditioners are better than the standard RAS
preconditioner for both discretisations. Finally, we also plot the convergence of the error of the different discretisations in Figure~\ref{fig:Poiseuille_T_shape_precontioner_metis}. And again, in all cases the MRAS preconditioner~\eqref{eq:dd_MRAS} shortens the plateau region.
\begin{figure}[!ht]
\centering
    \subfloat[Taylor-Hood]{%
      \includegraphics[width=0.45\textwidth]{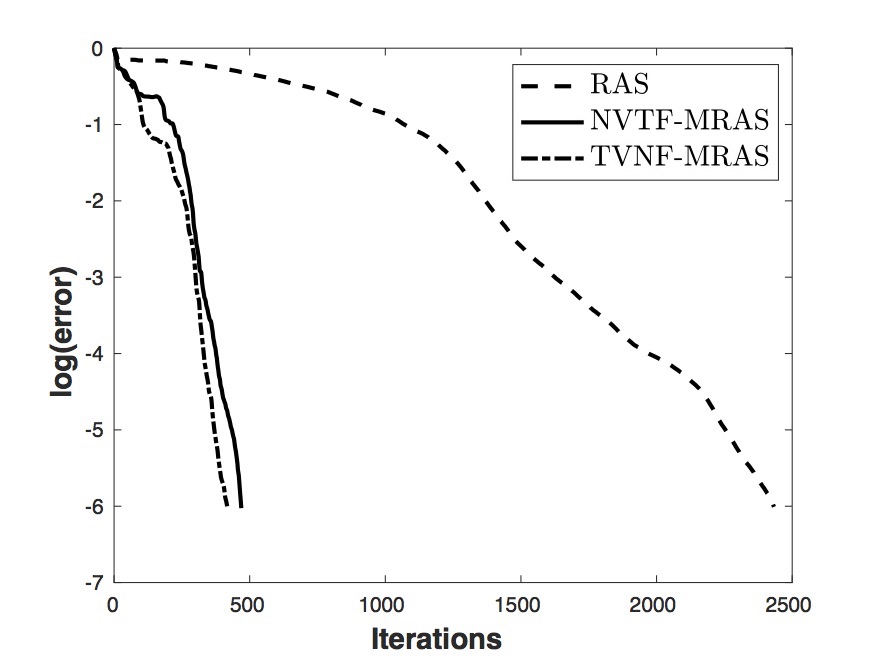}
    }
    \subfloat[hdG]{%
      \includegraphics[width=0.45\textwidth]{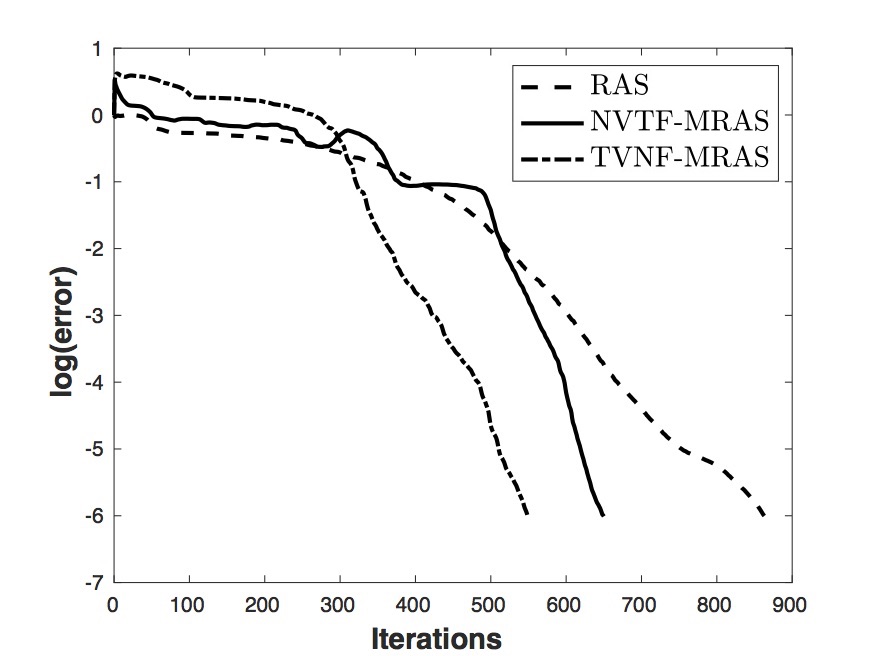}
    }
		\caption{Convergence of error for METIS decomposition in the 800 subdomains case - the T-shaped domain problem}
  \label{fig:Poiseuille_T_shape_precontioner_metis}
  \end{figure}

\newpage


\section{Conclusion}

In this paper we introduced a hdG method for the Stokes equations that naturally discretises non standard boundary value problems such as those with TVNF and NVTF boundary conditions. This approach can be extended naturally to the case of incompressible, or nearly incompressible, elasticity. We proved the well-posedness and convergence with respect to the norm~\eqref{eq:TVNF_norm} of this method and in the numerical experiments from Section~\ref{sec:hdG_numerics} we validated the theory and observed the optimal convergence.

To solve the discretised problem we introduced two different kinds of preconditioners with non standard boundary conditions whose optimality has been proved by algebraic techniques. We compared the newly introduced preconditioners to the more standard RAS preconditioner and numerical tests from Section~\ref{sec:dd_numerics} clearly show their superiority for different test cases in two space
dimensions. Moreover the hdG discretisation has an important advantage over Taylor-Hood as the RAS preconditioner already performs far better.

We observed, as expected, that the Schwarz preconditioners are not scalable with respect to the number of subdomains. However, this can be fixed by using an appropriate coarse spaces~\cite[Chapter 4]{MR3450068}. A suitable choice of a coarse space will be a subject of future research.

\section*{Acknowledgements}
This research was supported supported by the Centre for Numerical Analysis and Intelligent Software (NAIS). We thank Fr\'{e}d\'{e}ric Hecht from Laboratory J.L. Lions for comments that greatly improved the FreeFem++ codes.


\thispagestyle{empty}
  \bibliographystyle{alpha}
\bibliography{HDG}

\end{document}